\documentclass[11pt]{article}
\usepackage{amssymb}
\usepackage{amsmath}
\usepackage{amsthm}
\usepackage{authblk} 
\usepackage{array}
\usepackage{comment}
\usepackage{float}
\usepackage{mathpazo}
\usepackage[scr]{rsfso}
\usepackage{setspace}
\usepackage{subcaption}
\usepackage{tikz}
\usepackage{pgfplots}

\usepackage[textwidth=30mm]{todonotes}

\pgfplotsset{my style/.append style={axis x line=middle, axis y line=middle, xlabel={$x$}, ylabel={$y$}, axis equal }}
\makeatother

\usepackage[mathscr]{euscript}
\usepackage[top=3cm,bottom=3cm,right=2.8cm,left=2.8cm,headsep=0.4in]{geometry}

\theoremstyle{plain}
\newtheorem{theorem}{Theorem}[section]
\newtheorem{proposition}[theorem]{Proposition}
\newtheorem{lemma}[theorem]{Lemma}

\newtheorem{corollary}[theorem]{Corollary}

\newtheorem{remark}[theorem]{Remark}
\newtheorem{exl}[theorem]{Example}

\theoremstyle{definition}

\newcommand{\ep}{\varepsilon}
\newcommand{\E}{\mathbb{E}}
\newcommand{\N}{\mathbb{N}}
\newcommand{\R}{\mathbb{R}}
\newcommand{\PP}{\mathbb{P}}
\newcommand{\I}{{\rm I}}

\newcommand{\II}{{\rm II}}

\allowdisplaybreaks

\usetikzlibrary{arrows,decorations.markings,positioning,automata}
\usetikzlibrary{calc}
\tikzstyle{state}=[circle,thick,draw=black!80,minimum size=30pt]
\tikzstyle{rand}=[circle,thick,draw=black!80,fill=black,inner sep=0pt,minimum size=7pt]

\begin{document}
\author[a]{Galit Ashkenazi-Golan}
\author[b]{J\'{a}nos Flesch} 
\author[c]{Arkadi Predtetchinski}
\author[d]{Eilon Solan}

\affil[a]{London School of Economics and Political Science,
Houghton Street
London
WC2A 2AE, UK}
\affil[b]{Department of Quantitative Economics,
Maastricht University, P.O.Box 616, 6200 MD, The Netherlands}
\affil[c]{Department of Economics, Maastricht University, P.O.Box 616, 6200 MD,
The Netherlands}
\affil[d]{School of Mathematical Sciences, Tel-Aviv University, Tel-Aviv, Israel, 6997800}

\title{Regularity of the minmax value and equilibria in multiplayer Blackwell games\thanks{Ashkenazi-Golan acknowledges the support of the Israel Science Foundation, grants \#217/17 and \#722/18, and the NSFC-ISF Grant \#2510/17. Solan acknowledges the support of the Israel Science Foundation, grant \#217/17. This work has been partly supported by COST Action CA16228 European Network for Game Theory.}}
\maketitle

\begin{abstract}
\noindent 

A real-valued function $\varphi$ that is defined over all Borel sets of a topological space is \emph{regular} if for every Borel set $W$, $\varphi(W)$ is the supremum of $\varphi(C)$, over all closed sets $C$ that are contained in $W$, and the infimum of $\varphi(O)$, over all open sets $O$ that contain $W$.

We study Blackwell games with finitely many players. We show that when each player has a countable set of actions and the objective of a certain player is represented by a Borel winning set, that player's minmax value is regular.

We then use the regularity of the minmax value to establish the existence of $\ep$-equilibria in two distinct classes of Blackwell games. One is the class of $n$-player Blackwell games where each player has a finite action space and an analytic winning set, and the sum of the minmax values over the players exceeds $n-1$. The other class is that of Blackwell games with bounded upper semi-analytic payoff functions, history-independent finite action spaces, and history-independent minmax values. 
%The latter class subsumes all Blackwell games with history-independent finite action spaces and bounded and Borel-measurable and tail-measurable payoff functions. 
For the latter class, we obtain a characterization of the set of equilibrium payoffs.
\end{abstract}

\noindent\textbf{Keywords:} Blackwell games, determinacy, value, equilibrium, regularity.\medskip

\noindent\textbf{AMS classification code:} Primary: \textsc{91A44} (Games involving topology, set theory, or logic). Secondary: \textsc{91A20} (Multistage and repeated games).\medskip

\section{Introduction}
Blackwell games (Blackwell \cite{Blackwell69}) are dynamic multiplayer simultaneous-move games where the action sets of the players may be history dependent, and the payoff function is an arbitrary Borel-measurable function of the play. When the payoff function of a player is given by the characteristic function of a given set $W$, we say that $W$ is the \emph{winning set} of the player. These games subsume several familiar classes of dynamic games: repeated games with the discounted payoff or the limiting average payoff (e.g., Sorin \cite{Sorin92}, Mailath and Samuelson \cite{{Mailath06}}), games with perfect information (e.g., Gale and Stewart \cite{Gale53}), and graph games arising in the computer science applications (e.g., Apt and Gr\"{a}del \cite{AptGradel12}, Bruy\`{e}re \cite{Bruyere17, Bruyere21}, Chatterjee and Henzinger \cite{Chatterjee12}).

While two-player zero-sum Blackwell games and Blackwell games with perfect information are quite well understood (see, e.g., Martin \cite{Martin75, Martin98}, Mertens \cite{Mertens86}, Kuipers, Flesch, Schoenmakers, and Vrieze \cite{Kuipers21}), general multiplayer nonzero-sum Blackwell games have so far received relatively little attention.

The goal of this paper is to introduce a new technique to the study of multiplayer Blackwell games: regularity of the minmax value, along with a number of related approximation results. In a nutshell, the technique amounts to the approximation of the minmax value of a winning Borel set using a closed subset. This approach allows us to establish existence of $\ep$-equilibria in two distinct classes of Blackwell games.\medskip

\noindent\textsc{Regularity and approximation results:} A real-valued function $\varphi$ that is defined over all Borel sets of a certain space is \emph{inner regular} if for every Borel set $W$, $\varphi(W)$ is the supremum of $\varphi(C)$, over all closed sets $C$ that are contained in $W$. The function $\varphi$ is \emph{outer regular} if for every Borel set $W$ it is the infimum of $\varphi(O)$, over all open sets $O$ that contain $W$.
The function $\varphi$ is \emph{regular} if it is both inner regular and outer regular.
Borel probability measures on metric spaces are one example of a regular function (see, e.g., Kechris \cite[Theorems 17.10 and 17.11]{Kechris95}).

When restricted to two-player zero-sum Blackwell games with finite action sets and Borel-measurable winning set for Player~1, the value function is known (Martin \cite{Martin98}) to be regular. This result was extended to two-player zero-sum stochastic games by Maitra, Purves, and Sudderth \cite{Maitra92}.

We show that in multiplayer Blackwell games with countable action sets and Borel winning sets, the minmax value of all players is regular. We thus extend the regularity result of Martin \cite{Martin98} in terms of both the number of actions (countable versus finite) and the number of players (finite versus two).

A related approximation result concerns the case when a player's objective is represented by a bounded Borel-measurable payoff function. Denote by $v_i(f)$ player~$i$'s minmax value when her payoff function is $f$. We show that $v_i(f)$ is the supremum of $v_i(g)$ over all bounded limsup functions $g \leq f$, and the infimum of $v_i(g)$ over all bounded limsup function $g \geq f$. A \emph{limsup function} is a function that can be written as the limit superior of a sequence of rewards assigned to the nodes of the game tree. 
This too, is an extension of results by Maitra, Purves, and Sudderth \cite{Maitra92} and Martin \cite{Martin98} for two-player games to multiplayer games. If, moreover, the player's minmax value is the same in every subgame, one obtains an approximation from below by an upper semi-continuous function, and an approximation from above by a lower semi-continuous function.\medskip

\noindent\textsc{Existence of $\ep$-equilibria:} The main contribution of the paper is the application of the regularity of the minmax value to the problem of existence of an $\ep$-equilibrium in multiplayer Blackwell games. We establish the existence in two distinct classes of Blackwell games.

One is the class of $n$-player Blackwell games with bounded upper semi-analytic payoff functions, history-independent finite action spaces, and history-independent minmax values. The latter assumption means that every player's minmax value is the same in each subgame. Under these assumptions, for each $\ep > 0$, there is an $\ep$-equilibrium with a pure path of play. 

A prominent sufficient condition for the minmax value to be history-independent the is that payoff be tail-measurable. Roughly speaking, tail-measurability amounts to the requirement that the payoff is unaffected by a change of the action profile in any finite number of stages. We thus obtain the existence of $\ep$-equilibria in Blackwell games with history-independent finite action spaces and bounded, upper semi-analytic, and tail-measurable payoff functions. 

The second class of games for which we derive an existence result is $n$-player Blackwell games where each player has a finite action space at each history, her objective is represented by an analytic winning set, and the sum of the minmax values over the players exceeds $n-1$. Under these conditions we show that there exists a play that belongs to each player's winning set; any such play induces a $0$-equilibrium. At the heart of the proof is an approximation of each player's minmax value by the minmax value of a closed subset of the player's winning set.  

The key idea of the proof of the first result is to consider an auxiliary Blackwell game with winning sets, where the winning set of player $i$ is the set of player $i$'s $\ep$-individually rational plays: the plays that yield player $i$ a payoff no smaller then her minmax value minus $\ep$. We show that, in the thus-defined auxiliary Blackwell game, each player's minmax value equals 1, and apply the second\textbf{} result.

%One example of a tail-measurable payoff function is the limiting average reward in a setting of repeated games (see, e.g., Mertens and Neyman \cite{Mertens81}, Levy and Solan \cite{Levy20}). Classical computer science evaluation criteria such as B\"{u}chi, co-B\"{u}chi, parity, M\"{u}ller, and Streett objectives are tail-measurable (see, e.g., Apt and Gr\"{a}del \cite{AptGradel12}, Chatterjee \cite{Chatterjee07}, and Chatterjee and Henzinger \cite{Chatterjee12}). 

The question whether $\ep$-equilibria exist in multiplayer Blackwell games is a largely uncharted territory. An important benchmark is the result of Mertens and Neyman (see Mertens \cite{Mertens86}): all games of perfect information with bounded Borel-measurable payoff functions admit an $\ep$-equilibrium for every $\ep > 0$. Zero-sum Blackwell games (where at least one of the two players has a finite set of actions) are known to be determined since the seminal work of Martin \cite{Martin98}. Shmaya \cite{Shmaya11} extends the latter result by showing the determinacy of zero-sum games with eventual perfect monitoring, and Arieli and Levy \cite{Arieli15} extend Shmaya's result to stochastic signals. 

Only some special classes of multiplayer dynamic games have been shown to have an $\ep$-equilibrium. These include stochastic games with discounted payoffs (see, e.g., the survey by Ja\'{s}kiewicz and Nowak \cite{Nowak16}), two-player stochastic games with the limiting average payoff (Vieille \cite{Vieille00I,Vieille00II}), and graph games with classical computer science objectives (e.g., Secchi and Sudderth \cite{Sechi}, Chatterjee \cite{Chatterjee04,Chatterjee05}, Bruy\`{e}re \cite{Bruyere21}, Ummels, Markey, Brenguier, and Bouyer \cite{Ummels15}). 

A companion paper (\cite{AFPS}) establishes the existence of $\ep$-equilibria in Blackwell games with countably many players, finite action sets, and bounded, Borel-measurable, and tail-measurable payoff functions. %This is a generalization of one of the results we present here (namely of Corollary \ref{cor:eq}). 
The present paper departs from \cite{AFPS} in two dimensions. Firstly, it invokes a new proof technique, the regularity of the minmax value. Secondly, it makes different assumptions on the primitives. The second of our two existence results (Theorem \ref{theorem:sumofprob}) has, in fact, no analogue in \cite{AFPS}. The first (Theorem \ref{theorem:minmax_indt}) applies to a larger class of payoff functions than does the main result in \cite{AFPS}: it only requires players' minmax values to be history-independent. While tail-measurability of the payoff functions is a sufficient condition for history-independence of the minmax values, it is by no means a necessary condition. Furthermore, Borel-measurability imposed in \cite{AFPS} is relaxed here to upper semi-analyticity. On the other hand, \cite{AFPS} has a countable rather than a finite set of players, something that the methods developed here do not allow for.

%As one of the steps in our study, we establish an approximation result for a bounded and tail-measurable payoff functions: $v_i(f)$ is the supremum of $v_i(g)$, over all bounded  upper semicontinuous functions $g$ and the infimum of $v_i(g)$, over all bounded lower semicontinuous functions $g$.\medskip

\noindent\textsc{Characterisation of equilibrium payoffs:} An equilibrium payoff is an accumulation point of the expected payoff vectors of $\ep$-equilibria, as $\ep$ tends to $0$. We establish a characterisation of equilibrium payoffs in games with bounded upper semi-analytic payoff functions, history-independent finite action spaces, and history-independent minmax values. 

In repeated games with patient players the folk theorem asserts that under proper conditions, the set of limiting average equilibrium payoffs (or the limit set of equilibrium payoffs, as the discount factor goes to 0 or the horizon increases to infinity) is the set of all vectors that are individually rational and lie in the convex hull of the range of the stage payoff function (see, e.g., Aumann and Shapley \cite{Aumann94}, Sorin \cite{Sorin92}, or Mailath and Samuelson \cite{{Mailath06}}). Our result identifies the set of equilibrium payoffs of a Blackwell game as the set of all vectors that lie in the convex hull of the set of feasible and individually rational payoffs. The intuition for this discrepancy is that in standard repeated games, a low payoff in one stage can be compensated by a high payoff in another stage, therefore payoff vectors that are convex combinations of the stage payoff function can be equilibrium payoffs as long as this convex combination of payoffs is individually rational. In particular, these combinations can place some positive weight on payoff vectors that are not individually rational. In Blackwell games, however, the payoff is obtained only at the end of the game, hence only plays that generate individually rational payoffs can be taken into account when constructing equilibria.

Our characterization of the set of equilibrium payoffs is related to the rich literature on the folk theorem, and the study of the minmax value is instrumental to this characterizaion (see, e.g., the folk theorems in Fudenberg and Maskin \cite{Fudenberg86}, Mailath and Samuelson \cite{Mailath06}, or H\"{o}rner, Sugaya, Takahashi, and Vieille \cite{Horner11}). The minmax value of a player would often be used in the proofs of equilibrium existence to construct suitable punishments for a deviation from the supposed equilibrium play (as is done, for instance, in Aumann and Shapley \cite{Aumann94}, Rubinstein \cite{Rubinstein94}, Fudenberg and Maskin \cite{Fudenberg86}, and Solan \cite{Solan01}).\medskip

The paper is structured as follows. Section \ref{secn.games} describes the class of Blackwell games. Section \ref{secn.approx} is devoted to the regularity of the minmax value and related approximation theorems. Section \ref{secn.appl} applies these tools to the problem of existence of equilibrium. Section~\ref{secn.folk} is devoted to the characterisation of equilibrium payoffs. Section \ref{secn.tail} discusses the implications of the results for games with tail-measurable payoffs. Section \ref{secn.disc} contains a discussion, concluding remarks, and open questions.

\section{Blackwell games}\label{secn.games}
\noindent\textbf{Blackwell games:} An $n$-\textit{player Blackwell game} is a tuple $\Gamma = (I, A, H, (f_{i})_{i \in I})$. The elements of $\Gamma$ are as follows.

The set of players is $I$, a finite set of cardinality $n$. For a player $i \in I$ we write $-i$ to denote the set of $i$'s opponents, $I \setminus \{i\}$. 

The set $A$ is a countable set and $H \subseteq \cup_{t \in \N}A^{t}$ is the game tree (throughout the paper $\N = \{0,1,\ldots\}$). Elements of $H$ are called histories. The set $H$ is assumed to have the following properties: (a) $H$ contains the empty sequence, denoted $\oslash$; (b) a prefix of an element of $H$ is an element of $H$; that is, if for some $h \in \cup_{t \in \N}A^{t}$ and $a \in A$ the sequence $(h,a)$ is an element of $H$, so is $h$; (c) for each $h \in H$ there is an element $a \in A$ such that $(h, a) \in H$; we define $A(h) := \{a \in A: (h,a) \in H\}$; and (d) for each $h \in H$ and each $i \in I$ there exists a set $A_{i}(h)$ such that $A(h) = \prod_{i \in I}A_{i}(h)$. The set $A_{i}(h)$ is called player $i$'s set of actions at history $h$, and $A(h)$ the set of action profiles at $h$. 

Conditions (a), (b), and (c) above say that $H$ is a pruned tree on $A$. 
Condition (c) implies that the game has infinite horizon.
Let $H_{t} := H \cap A^{t}$ denote the set of histories in stage $t$.     

An infinite sequence $(a_0,a_1,\ldots) \in A^{\N}$ such that $(a_0,\ldots,a_t) \in H$ for each $t \in \N$ will be called a \emph{play}. The set of plays is denoted by $[H]$.
This is the set of infinite branches of $H$. For $h \in H$ let $O(h)$ denote the set of all plays of $\Gamma$ having $h$ as a prefix. We endow $[H]$ with the topology generated by the basis consisting of the sets $\{O(h):h \in H\}$. The space $[H]$ is Polish. For $t \in \N$ let $\mathcal{F}_{t}$ be the sigma-algebra on $[H]$ generated by the sets $\{O(h):h \in H_{t}\}$. The Borel sigma-algebra of $[H]$ is denoted by $\mathscr{B}$. It is the minimal sigma-algebra containing the topology. A subset $S$ of $[H]$ is \emph{analytic} if it is the image of a continuous function from the Baire space $\N^\N$ to $[H]$. Each Borel set is analytic.

Each analytic set is universally measurable. Recall that a set $S \subseteq [H]$ is said to be universally measurable if (Kechris \cite[Section 17.A]{Kechris95}), for every Borel probability measure $\PP$ on $[H]$, there exist Borel sets $B,Z \in \mathscr{B}$ such that $S \bigtriangleup B \subseteq Z$ and $\PP(Z) = 0$; here $S \bigtriangleup B = (S \setminus B) \cup (B \setminus S)$ is the symmetric difference of the sets $S$ and $B$. 

The last element of the game is a vector $(f_i)_{i \in I}$, where $f_i : [H] \to \mathbb{R}$ is player $i$'s \emph{payoff function}. The most general class of payoff functions we allow for are bounded upper semi-analytic functions. A function $f_i : [H] \to \mathbb{R}$ is said to be \emph{upper semi-analytic} if, for each $r \in \R$, the set $\{p \in [H]: r \leq f(p)\}$ is analytic. In particular, the indicator function $1_{S}$ of a subset $S \subseteq [H]$ of plays is upper semi-analytic if and only if $S$ is an analytic set. Each Borel-measurable function in upper semi-analytic. Note that a bounded upper semi-analytic function is universally measurable, i.e., for each open set $U \subseteq \R$, the set $f^{-1}(U) \subseteq [H]$ is universally measurable (see, e.g., Chapter 7 in Bertsekas and Shreve \cite{Bertsekas96}).

The play of the game starts at the empty history $h_0 = \oslash$. Suppose that by a certain stage $t \in \N$ a history $h_t \in H_t$ has been reached. Then in stage $t$, the players simultaneously choose their respective actions; thus player $i \in I$ chooses an action $a_{i,t} \in A_{i}(h_{t})$. This results in the stage $t$ action profile $a_{t} = (a_{i,t})_{i \in I} \in A(h_{t})$. Once chosen, the actions are revealed to all players, and the history $h_{t+1} = (h_{t},a_{t})$ is reached. The result of the infinite sequence of choices is the play $p = (a_0,a_1,\ldots)$, an element of $[H]$. Each player $i \in I$ receives the corresponding payoff $f_i(p)$.  

Given a Blackwell game $\Gamma$ and a history $h \in H$, the \textit{subgame} of $\Gamma$ starting at $h$ is the Blackwell game $\Gamma_h = (I, A, H_h, (f_{i,h})_{i \in I})$. The set $H_h$ of histories of $\Gamma_h$ consists of finite sequences $g \in \bigcup_{t \in \N} A^{t}$ such that $hg \in H$, where $hg$ is the concatenation of $h$ and $g$. The payoff function $f_{i,h} : [H_h] \to \R$ is the composition $f_i \circ s_h$, with $s_h : [H_h] \to [H]$ given by $p \mapsto hp$, where $hp$ is the concatenation of $h$ and $p$. Note that $\Gamma_\oslash$ is just the game $\Gamma$ itself. 

The Blackwell game $\Gamma$ is said to have \textit{history-independent action sets} if $A_{i}(h) = A_{i}(\oslash)$ for each history $h \in H$ and each player $i \in I$; the common action set is simply denoted by $A_i$. If $\Gamma$ has history-independent action sets, then the set of its histories is $H = \cup_{t \in \N}A^{t}$, and the set of plays in $\Gamma$ is $[H] = A^{\N}$. A Blackwell game with history-independent action sets can be described as a tuple $(I,(A_{i},f_{i})_{i \in I})$. 
\medskip

\noindent\textbf{Strategies and expected payoffs:} A strategy for player $i\in I$ is a function $\sigma_i$ assigning to each history $h \in H$ a probability distribution $\sigma_{i}(h)$ on the set $A_{i}(h)$. The set of player $i$'s strategies is denoted by $\Sigma_i$. We also let $\Sigma_{-i} := \prod_{j \in -i} \Sigma_{j}$ and $\Sigma := \prod_{i \in I} \Sigma_{i}$. Each strategy profile $\sigma=(\sigma_i)_{i\in I}$ induces a unique probability measure on the Borel sets of $[H]$, denoted $\PP_{\sigma}$. The corresponding expectation operator is denoted $\E_{\sigma}$. In particular, $\E_{\sigma}[f_{i}]$ denotes an expected payoff to player $i$ in the Blackwell game under the strategy profile $\sigma$. It is well defined under the maintained assumptions, namely boundedness and upper semi-analyticity of $f_{i}$.  

Take a history $h \in H_t$ in stage $t$. A strategy profile $\sigma \in \Sigma$ in $\Gamma$ induces the strategy profile $\sigma_h$ in $\Gamma_h$ defined as $\sigma_h(g)=\sigma(hg)$ for each history $g \in H_h$. Let us define $\E_{\sigma}(f_i \mid h)$ as the expected payoff to player $i$ in the Blackwell game $\Gamma_h$ under the strategy profile $\sigma_h$: that is,
$\E_{\sigma}(f_i \mid h) := \E_{\sigma_h}(f_{i,h})$. Note that $\E_{\sigma}(f_i \mid h)$, when viewed as an $\mathcal{F}_{t}$-measurable function on $[H]$, is a conditional expectation of $f_i$ with respect to the measure $\PP_{\sigma}$ and the sigma-algebra $\mathcal{F}_{t}$; whence our choice of notation.\medskip

\noindent\textbf{Minmax value:} Consider a Blackwell game $\Gamma$, and suppose that player $i$'s payoff function $f_i$ is bounded and upper semi-analytic. Player $i$'s \textit{minmax value} is defined as 
\[v_i(f_i) := \inf_{\sigma_{-i} \in \Sigma_{-i}}\sup_{\sigma_{i} \in \Sigma_{i}} \mathbb{E}_{\sigma_{-i},\sigma_{i}}(f_i).\]
Whenever $f_i = 1_{W_i}$ is an indicator of an analytic set $W_i \subseteq [H]$ we write $v_i(W_i)$ for $v_i(1_{W_i})$. 

Player $i$'s minmax value is said to be \textit{history-independent} if her minmax value in the subgame $\Gamma_h$ equals that in the game $\Gamma$, for each history $h \in H$.

%\noindent\textbf{Tail-measurability:} 

\section{Regularity and approximation theorems}\label{secn.approx}
In this section we state the regularity property of the minmax: the minmax value of a Borel winning set can be approximated from below by the minmax value of closed subset and from above by the minmax value of an open superset. We also describe two related approximation results: the minmax value of a bounded Borel-measurable payoff function can be approximated from below and from above by limsup functions. If, in addition, the minmax values are history-independent, then one can choose the approximation from below to be upper semicontinuous, and the approximation from above to be lower semicontinuous.
The proofs of all results are detailed in the appendix. 
%We start with the regularity of the minmax \color{red} and maxmin values\color{black}.\medskip

\begin{theorem} {\rm (Regularity of the minmax value)}\label{thrm:reg}
Consider a Blackwell game. Suppose that player $i$'s objective is given by a winning set $W_i \subseteq [H]$. Suppose that $W_i$ is Borel. Then 
\begin{align*}
v_i(W_i) &= \sup\{v_i(C):C\subseteq W_i, C \text{ is closed}\}=\inf\{v_i(O):O\supseteq W_i, O \text{ is open}\}.
\end{align*}
\end{theorem}

One implication of Theorem~\ref{thrm:reg} concerns the complexity of strategies of player~$i$ that ensures that her probability of winning is close to her minmax value. 
Suppose, for example, that $v_i(W_i) = \frac{1}{2}$.
Then for every strategy profile $\sigma_{-i}$ of the opponents of player~$i$ and every $\ep > 0$, she has a response $\sigma_i$ such that $\PP_{\sigma_{-i},\sigma_i}(W_i) \geq \frac{1}{2}-\ep$. The strategy profile $\sigma_{-i}$ and the winning set $W_i$ may be  complex, and accordingly the good response $\sigma_i$ may be complex as well.
However, take now a closed subset $C \subseteq W_i$ such that $v_i(C) > v_i(W_i) - \ep = \frac{1}{2}-\ep$. The complement of $C$, denoted $C^c$, is open, hence it is the union of basic open sets; that is, it can be presented as a union $C^c = \bigcup_{h \in H'} O(h)$,
for some subset $H' \subseteq H$ of histories. A strategy $\sigma'_i$ that satisfies  
$\PP_{\sigma_{-i},\sigma_{i}'}(C_i) \geq \frac{1}{2}-\ep$ must aim at avoiding $C^c$, that is, at avoiding histories in $H'$. In that sense, $\sigma'$ may have a simple structure.

\begin{exl}\label{exl.io}\rm 
Here we consider a Blackwell game where the same stage game is being played at every stage. The stage game specifies a stage winning set for each player. A player's objective in the Blackwell game is to win the stage game infinitely often.

Thus let $\Gamma = (I,(A_{i},1_{W_{i}})_{i \in I})$ be a Blackwell game with history-independent countable action sets, where player $i$'s winning set is
\[W_i = \{(a_0,a_1,\ldots) \in A^\N:a_t \in U_i\text{ for infinitely many }t \in \N\};\]
here $U_{i}$, called player $i$'s \textit{stage winning set}, is a given subset of $\prod_{i \in I}A_{i}$. If $a_t \in U_i$, we say that player $i$ \textit{wins stage $t$}. Thus, player~$i$'s objective is to win infinitely many stages of the Blackwell game. The set $W_{i}$ is a $G_\delta$-set, i.e., an intersection of countably many open subsets of $A^\N$. 

Fix a player $i \in I$. 
Let 
\begin{equation}\label{eqn.stageminmax}
d_{i} := \inf_{x_{-i} \in X_{-i}} \sup_{x_i \in X_i} \PP_{x_{-i},x_i}(U_i) 
\end{equation} 
be player~$i$'s minmax value in the stage game. As follows from the arguments below, $v_{i}(W_{i})$ is either $0$ or $1$, and it is $1$ exactly when $d_{i} > 0$. In either case, there  are  intuitive approximations of player $i$'s wining sets by a closed set from below and an open set from above.   

First assume that $d_{i} > 0$. Take an $\ep > 0$. Let us imagine that player $i$'s objective is not merely to win infinitely many stages in the course of the Blackwell game, but to make sure that she wins at least once in every block of stages $t_{n},\ldots,t_{n+1}-1$, where the sequence of stages $t_0 < t_1 < \cdots$ is chosen to satisfy 
\[(1 - \tfrac{1}{2}d_{i})^{t_{n+1} - t_{n}} < 2^{-n-1}\cdot\ep\]
for each $n \in \N$. This, more demanding condition, defines an approximating set. Formally, define 
\[C_{i} := \bigcap_{n \in \N}\bigcup_{t_{n} \leq k < t_{n+1}} \{(a_0,a_1,\ldots) \in A^{\N}: a_k \in U_i\}.\]

As the intersection of closed sets, $C_{i}$ is a closed subset of $W_{i}$. Moreover, $1 - \ep \leq v_{i}(C_{i})$. To see this, fix any strategy $\sigma_{-i}$ for $i$'s opponents. At any history $h$, player $i$ has a mixed action $\sigma_{i}(h)$ that, when played against $\sigma_{-i}(h)$, guarantees a win at history $h$ with probability of at least $\tfrac{1}{2}d_{i}$. Thus, under the measure $\PP_{\sigma_{-i},\sigma_{i}}$ the probability for player $i$ not to win at least once in a block of stages $t_{n},\ldots,t_{n+1}-1$ is at most $2^{-n-1}\cdot\ep$, for any history of play up to stage $t_{n}$. And hence the probability that there is a block within which player $i$ does not win once is at most $\ep$.       

Suppose that $d_{i} = 0$. Let us imagine that player $i$'s objective is merely to win the stage game at least once. This modest objective defines the approximating set:  
\[O_{i} = \bigcup_{t \in \N}\{(a_{0},a_{1},\ldots)\in A^\N: a_{t} \in U_i\}.\]
As the union of open sets,
$O_{i}$ is an open set containing $W_{i}$. Moreover, $v_{i}(O_{i}) \leq \ep$. To see this, let $\sigma_{-i}$ be the strategy for $i$'s opponents such that, at any stage $t \in \N$ and any history $h \in A^{t}$ of stage $t$, the probability that the action profile $a_{t}$ is an element of $U_{i}$ is not greater than $2^{-t-1}\cdot\ep$ regardless of the action of player $i$. Then, for any player $i$'s strategy $\sigma_{i}$, the probability that $i$ wins at least once is not greater than $\ep$. $\Box$

%To prove that $\bigcap_{i \in I}W_{i}$ not empty, define $C = \bigcap_{i \in I}C_{i}$ and show that $C$ is not empty. Let $x \in \prod_{i \in I}\Delta(A_i)$ be any Nash equilibrium of the stage game, and let $\sigma$ be the strategy profile assigning $x$ to each history of $\Gamma$. For each player $i \in I$, under the measure $\PP_{\sigma}$ the probability for a player $i$ to win any particular stage is at least $d_i$ irrespective of the history. And hence an argument similar to that in the preceding paragraph shows that $1 - \ep \leq \PP_{\sigma}(C_{i})$. And thus $0 < 1 - |I|\ep \leq \PP_{\sigma}(C)$.\medskip
\end{exl}

We turn to two related approximation results for Blackwell games with Borel payoff functions. A function $f : [H] \to \mathbb{R}$ is said to be a \textit{limsup function} if there exists a function $u : H \to \mathbb{R}$ such that for each play $(a_0,a_1,\ldots) \in [H]$, 
\[f(a_0,a_1,\ldots) = \limsup_{t \to \infty} u(a_0,\dots,a_t).\]
The function $f : [H] \to \mathbb{R}$ is a \textit{liminf function} if $-f$ is a limsup function. 

Limsup and liminf payoff functions are ubiquitous in the literature on infinite dynamic games. At least since the work of Gillette \cite{Gillette57}, the so-called limiting average payoff (that is, the limit superior or the limit inferior of the average of the stage payoffs) is a standard specification of the payoffs in a stochastic game (see for example Mertens and Neyman \cite{Mertens81}, or Levy and Solan \cite{Levy20}). Stochastic games with limsup payoff functions have been studied in Maitra and Sudderth \cite{Maitra93}.  

Limsup functions have relatively ``low" set-theoretic complexity. Various characterizations of the limsup functions can be found in Hausdorff \cite{Hausdorff05}. In particular, $f$ is a limsup function if and only if, for each $r \in \R$, the set $\{p \in [H]: r \leq f(p)\}$ is a $G_{\delta}$-set.    

We now state a result on the approximation of the minmax value for Blackwell games where a player's objective is represented by a bounded Borel-measurable payoff function.

\begin{theorem}\label{thrm:regfunc} 
Consider a Blackwell game. Suppose that player $i$'s payoff function $f_i : [H] \to \R$ is bounded and Borel-measurable. Then:
\[\begin{aligned}
v_i(f_i) &= \sup\{v_i(g): g\text{ is a bounded limsup function and }g \leq f_i\}\\
&= \inf\{v_i(g): g \text{ is a bounded limsup function and }f_i \leq g\}.
\end{aligned}\]
\end{theorem}

Theorems~\ref{thrm:reg} and~\ref{thrm:regfunc} have been proven by Martin \cite{Martin98}
for the case $n=2$, see \cite[Theorem 5, and Remark (b)]{Martin98} and 
\cite[Remark (c)]{Martin98}. They have been extended to two-player stochastic games by 
Maitra and Sudderth \cite{Maitra98}. Theorems~\ref{thrm:reg} and~\ref{thrm:regfunc}
extend the known results in two respects. First, they allow for more than two players,
and second, they allow for countably many actions.

The proof of Theorem \ref{thrm:reg} combines and fine-tunes the arguments in Martin \cite{Martin98} and Maitra and Sudderth \cite{Maitra98}. The key element of the proof is a zero-sum perfect information game, 
denoted $G_i(f_i,c)$, where the aim of Player I is to ``prove" that the minmax value of $f_i$ is at least $c$. Roughly speaking, the game proceeds as follows. 
Player~I commences the game by proposing a fictitious continuation payoff, which one could think of as a payoff player $i$ hopes to attain, contingent on each possible stage $0$ action profile. The number $c$ serves as the initial threshold: player $i$'s minmax value of the proposed continuation payoffs is required to be at least $c$. Player II then chooses a stage $0$ action profile, and the corresponding continuation payoff serves as the new threshold. 
Player I then proposes a fictitious continuation payoff contingent on each possible stage $1$ action profile, 
and Player II chooses the stage $1$ action profile, etc. 
Player I wins if the sequence of continuation payoffs is ``justified" by the actual payoff on a play produced by Player II. Ultimately the proof rests on the determinacy of the game $G_i(f_i,c)$, which follows by Martin \cite{Martin75}.

The perfect information game $G_i(f_i,c)$ is a version of the games used in Martin \cite{Martin98}. The main difference is in the use of player $i$'s minmax value that constrains Player I's choice of fictitious continuation payoffs. The details of our proof are slightly closer to those in Maitra and Sudderth \cite{Maitra98}. Like them we invoke martingale convergence and the Fatou lemma.  

%The proof of the result for the maxmin value is similar to the proof of the theorem for the minmax value, so we only spell out the details of the latter. 
%We also remark that the result on the approximation of the maxmin follows at once from regularity of the value, (Martin \cite[Theorem 5]{Martin98}) once the action sets are assumed to be finite.

%It is also true that a liminf function could be used for an approximation from below and a limsup function for the approximation from above. Only minimal adaption of the proof would be needed.

Finally, we state an approximation result for a Blackwell game with history-independent minmax values. Recall that a function $g : A^\N \to \R$ is \emph{upper semicontinuous} if, for each $r \in \R$, the set $\{p \in [H]: r \leq f(p)\}$ is a closed set, and $g$ is \emph{lower semicontinuous} if $-g$ is upper semicontinuous. When $g = 1_B$ for some $B \subseteq A^\N$, $g$ is upper semicontinuous (resp.~lower semicontinuous) if and only if $B$ is closed (resp.~open).

\begin{theorem}\label{thrm:tailapprox}
Consider a Blackwell game. Suppose that player $i$'s payoff function $f_i$ is bounded and Borel-measurable, and player $i$'s minmax values are history-independent. Then
\[\begin{aligned}
v_i(f_i) 
&= \sup\{v_i(g): \text{g is a bounded upper semicontinuous function and }g \leq f_i\}\\
&= \inf\{v_i(g): \text{g is a bounded lower semicontinuous function and }f_i \leq g\}.
\end{aligned}\]
\end{theorem}

As the proof reveals, both the upper semicontinuous and the lower semicontinuous functions can be chosen to be two-valued. Recall that (Hausdorff \cite{Hausdorff05}) an upper semicontinuous function and a lower semicontinuous function are both a limsup and a liminf function. Consequently, in comparison to Theorem \ref{thrm:regfunc}, an additional assumption of history-independence of the minmax values in Theorem \ref{thrm:tailapprox} leads to a stronger approximation result. The latter condition cannot be dropped; see Section~\ref{secn.disc} for an example of a game with a limsup payoff function such that the minmax value cannot be approximated from below by an upper semicontinuous function.

\section{Existence of equilibria}\label{secn.appl}
In this section, we employ the results of the previous section to establish existence of $\ep$-equilibria in two distinct classes of Blackwell games. Theorem~\ref{theorem:sumofprob} concerns $n$-player Blackwell games where each player has a finite action space at each history, her objective is represented by an analytic winning set, and the sum of the minmax values over the players exceeds $n-1$. Theorem~\ref{theorem:minmax_indt} concerns for Blackwell games with bounded upper semi-analytic payoff functions, history-independent finite action spaces, and history-independent minmax values.

Consider a Blackwell game $\Gamma$ and let $\ep \geq 0$. A strategy profile $\sigma \in \Sigma$ is \emph{an $\ep$-equilibrium} of $\Gamma$ if for each player $i \in I$ and each strategy $\eta_i \in \Sigma_i$ of player $i$, 
\[\mathbb{E}_{\sigma_{-i},\eta_{i}}(f_i) \leq \mathbb{E}_{\sigma_{-i},\sigma_{i}}(f_i) + \ep.\]

We state our first existence result.

\begin{theorem}\label{theorem:sumofprob}
Consider an $n$-player Blackwell game $\Gamma = (I, A, H, (1_{W_{i}})_{i \in I})$. Suppose that for each player $i \in I$  player $i$'s action set $A_i(h)$ at each history $h \in H$ is finite, and that her winning set $W_i$ is analytic. If $v_1(W_1)+ \cdots +v_n(W_n) > n-1$, then the set $W_1 \cap \cdots \cap W_n$ is not empty. Consequently, $\Gamma$ has a 0-equilibrium. 
\end{theorem}

Note that any play $p \in W_1 \cap \cdots \cap W_n$ is in fact a 0-equilibrium, or more precisely, any strategy profile that requires all the players to follow $p$ is a 0-equilibrium, because it yields all players the maximal payoff 1.

The key step of the proof is the approximation of the minmax value of a player using a closed subset of her winning set. 
To prove Theorem~\ref{theorem:sumofprob} we need the following technical observation.

\begin{lemma}\label{lemma:intersect}
Let $(X,\mathscr{B},P)$ be a probability space, and let $Q_1,\ldots,Q_n\in \mathscr{B}$ be $n$ events. Then
\[P(Q_1\cap\cdots\cap Q_n)\geq P(Q_1)+\cdots+P(Q_n)-n+1.\]
\end{lemma}

\begin{proof}
For $n=1$ the statement is obvious, and for $n=2$ we have
\begin{equation}\label{indineq}
P(Q_1\cap Q_2)=P(Q_1)+P(Q_2)-P(Q_1\cup Q_2)\geq P(Q_1)+P(Q_2)-1.
\end{equation}
Assume that the statement holds for some $n-1$. Then for $n$ we have
\begin{align*}
P(Q_1\cap\cdots\cap Q_n)\,&=\,P((Q_1\cap\cdots\cap Q_{n-1})\cap Q_n)\\
&\geq\, P(Q_1\cap\cdots\cap Q_{n-1}) + P(Q_n)-1\\
&\geq\, \big(P(Q_1)+\cdots+P(Q_{n-1})-n+2\big)+ P(Q_n)-1\\
&=\,P(Q_1)+\cdots+P(Q_n)-n+1,
\end{align*}
where the first inequality follows from Eq.~\eqref{indineq} and the second by the induction hypothesis.
\end{proof}

\noindent\textbf{Proof of Theorem \ref{theorem:sumofprob}:} We first establish the theorem in the special case of Borel winning sets, and then generalize it to analytic winning sets.\smallskip

\noindent\textsc{Part I:} Suppose that for each $i \in I$ the set $W_{i} \subseteq [H]$ is Borel.

By Theorem \ref{thrm:reg} there are closed sets $C_1\subseteq W_1,\ldots,C_n\subseteq W_n$ such that $v_1(C_1)+\cdots+v_n(C_n) > n-1$. We show that the intersection $C_1 \cap\cdots\cap C_n$ is not empty.

Given $m \in \N$ consider the $n$-player Blackwell game $\Gamma^{m} = (I, A, H, (1_{C_{i}^{m}})_{i \in I})$, where player $i$'s winning set is defined by 
\[C_i^{m} := \bigcup\{O(h): h \in H_{m}\text{ such that }O(h) \cap C_{i} \neq \oslash\}.\]

The game $\Gamma^{m}$ essentially ends after $m$ stages: by stage $m$ each player $i$ knows whether the play is an element of her winning set $C_{i}^{m}$ or not. In $\Gamma^{m}$, player $i$ wins if after $m$ stages there is a continuation play that leads to $C_i$. 
Note that this continuation play might be different for different players. %Intuitively, player $i$ wins in $\Gamma^{m}$ if after $m$ stages it is still possible to end up in $C_i$. 

The set $C_i^{m}$ is a clopen set. For each $m \in \mathbb{N}$ and $i \in I$ we have the inclusion $C_i^{m} \supseteq C_i^{m+1}$ (winning in $\Gamma^{m+1}$ is more difficult than winning in $\Gamma^{m}$). Moreover, $\bigcap_{m \in \mathbb{N}} C_i^{m} = C_i$. Indeed, the inclusion $C_i^{m} \supseteq C_i$ is evident from the definition. Conversely, take an element $q$ of the set $[H] \setminus C_i$. Since $[H] \setminus C_i$ is an open set, there exists a history $h \in H$ such that $q \in O(h)$ and $O(h) \subseteq [H] \setminus C_{i}$. But then $q \in [H] \setminus C_i^{m}$, where $m$ is the length of the history $h$. 

Define $C^m := C_1^m \cap\cdots\cap C_n^m$. Thus $\{C^{m}\}_{m \in \mathbb{N}}$ is a nested sequence of closed sets converging to $C_1 \cap\cdots\cap C_n$. Note that, since by the assumption of the theorem $H$ is a finitely branching tree, the space $[H]$ is compact. Thus $C^{m}$ is a compact set.  Consequently, to prove that $C_{1} \cap \cdots \cap C_{n}$ is not empty, we only need to argue that $C^{m}$ is not empty for each $m \in \mathbb{N}$.

The game $\Gamma^m$ being finite, it has a $0$-equilibrium (Nash \cite{Nash50}), say $\sigma^m$. By the definition of $0$-equilibrium,
the equilibrium payoff is not less than the minmax value:
\[\PP_{\sigma_{-i}^m,\sigma_{i}^{m}} (C_i^m) \,=\, \sup_{\sigma_{i} \in \Sigma_{i}} \PP_{\sigma_{-i}^m,\sigma_{i}} (C_i^m)\, \geq\, \inf_{\sigma_{-i} \in \Sigma_{-i}} \sup_{\sigma_{i} \in \Sigma_{i}} \PP_{\sigma_{-i},\sigma_{i}} (C_i^m) = v_i(C_i^m).\]
Moreover, since $C_i^{m} \supseteq C_{i}$, it holds that $v_i(C_i^m) \geq v_i(C_i)$. We conclude that
\[\PP_{\sigma^m}(C_1^m)+ \cdots +\PP_{\sigma^m}(C_n^m) > n-1.\]
Finally, we apply Lemma \ref{lemma:intersect} to conclude that $\PP_{\sigma^m}(C^m) > 0$, hence $C^{m}$ is not empty.\smallskip

\noindent\textsc{Part II:} Now let $\Gamma$ be any game as in the statement of the theorem. Suppose by way of contradiction that $W_1 \cap \cdots \cap W_n$ is empty. By Novikov's separation theorem  (Kechris \cite[Theorem 28.5]{Kechris95}) there exist Borel sets $B_1, \ldots, B_n$ such that $W_{i} \subseteq B_{i}$ for each $i \in I$ and $B_1 \cap \cdots \cap B_n = \oslash$. But since $v_{i}(W_{i})\leq v_{i}(B_{i})$ for each $i \in I$, the game $\Gamma = (I, A, H, (1_{B_{i}})_{i \in I})$ satisfies the assumptions of the theorem, and Part I of the proof yields a contradiction. $\Box$\medskip

We state our second and main existence result. 

\begin{theorem}\label{theorem:minmax_indt}
Consider a Blackwell game $\Gamma = (I, A, H, (f_{i})_{i \in I})$. Suppose that for each player $i \in I$, player $i$'s action set $A_i(h)$ at each history $h \in H$ is finite, her payoff function $f_i$ is bounded and upper semi-analytic, and her minmax value is history-independent. Then for every $\ep>0$ the game admits an $\ep$-equilibrium.
\end{theorem}

The key idea behind the proof is to consider an auxiliary Blackwell game with winning sets, the winning set of a player consisting of that player's $\ep$-individually rational plays. We show that in the thus-defined auxiliary Blackwell game each player's minmax value equals 1, and apply Theorem \ref{theorem:sumofprob}.  

%The proof reduces the problem of existence of $\ep$-equilibrium in $\Gamma$ to that in a game where the player's objectives are represented by a particular tail-measurable winning set. The heart of the proof is Theorem~\ref{theorem:sumofprob} below, which states that in a game with tail-measurable winning sets, if each player's minmax value is $1$, then the players' winning sets have a point in common. This common point, complemented with a suitable threat of punishment, can be used to construct an $\ep$-equilibrium. To prove that the players' winning sets have a point in common one approximates them from below by a closed set.

Given $\ep > 0$ we define the set of \textit{player $i$'s $\ep$-individually rational plays}:
\[Q_{i,\ep}(f_i) := \{p\in [H] : f_i(p) \geq v_i(f_i) - \ep\}.\] 
Also define the set
\[U_{i,\ep}(f_i) := \{p\in [H] : f_i(p) \geq v_i(f_i) + \ep\}.\] 
Note that under the assumptions of Theorem \ref{theorem:minmax_indt} both sets are analytic.

\begin{proposition}\label{prop:v(Q)=1}
Consider a Blackwell game $\Gamma = (I, A, H, (f_{i})_{i \in I})$ and a player $i \in I$. Suppose that player $i$'s payoff function $f_i$ is bounded and upper semi-analytic, and that her minmax values are history-independent. Let $\ep > 0$. Then 
\begin{enumerate}
\item $v_i(Q_{i, \ep}(f_i)) = 1$. In fact, for each strategy profile $\sigma_{-i} \in \Sigma_{-i}$ of players $-i$ there is a strategy $\sigma_i \in \Sigma_{i}$ for player $i$ such that $\mathbb{P}_{\sigma_{-i},\sigma_{i}}(Q_{i, \ep}(f_i)) = 1$.
\item $v_i(U_{i, \ep}(f_i)) = 0$. In fact, there exists a strategy profile $\sigma_{-i} \in \Sigma_{-i}$ of players $-i$ such that for each strategy $\sigma_i \in \Sigma_{i}$ for player $i$ it holds that $\mathbb{P}_{\sigma_{-i},\sigma_{i}}(U_{i, \ep}(f_i)) = 0$.
\end{enumerate}

\begin{proof} 
\noindent\textsc{Claim 1:} It suffices to prove the second statement. Take a strategy profile $\sigma_{-i}$ of players $-i$. It is known that player $i$ has a strategy $\sigma_i$ that is an $\ep/2$-best response to $\sigma_{-i}$ in each subgame (see, for example, Mashiah-Yaakovi \cite[Proposition 11]{Ayala15}, or Flesch, Herings, Maes, and Predtetchinski \cite[Theorem 5.7]{JJJJ}), and therefore 
\[\mathbb{E}_{\sigma_{-i},\sigma_i}(f_i \mid h) \geq
v_i(f_{i,h}) -\ep/2\,=\,v_i(f_i)-\ep/2,\]
for each history $h\in H$. Since the payoff function $f_i$ is bounded, it follows that there is $d>0$ such that 
\[\PP_{\sigma_{-i},\sigma_i}(Q_{i, \ep}(f_i) \mid h) \geq d\]
for each $h \in H$. Indeed, it is easy to verify that one can choose 
\[d\,=\,\frac{\ep}{2(\sup_{p \in [H]}f_i(p)-v_i(f_i)+\ep)}.\]

Since $Q_{i,\ep}(f_i)$ is an analytic set, there is a Borel set $B$ such that $\PP_{\sigma_{-i},\sigma_i}(Q_{i, \ep}(f_i) \bigtriangleup B) = 0$, where $\bigtriangleup$ stands for the symmetric difference of two sets. It follows that $\PP_{\sigma_{-i},\sigma_i}(Q_{i, \ep}(f_i) \bigtriangleup B \mid h) = 0$, and consequently $\PP_{\sigma_{-i},\sigma_i}(B \mid h) \geq d$ for each history $h \in H$ that is reached under $\PP_{\sigma_{-i},\sigma_i}$ with positive probability. L\'{e}vy's zero-one law implies that $\PP_{\sigma_{-i},\sigma_i}(B) = 1$, and hence $\PP_{\sigma_{-i},\sigma_i}(Q_{i, \ep}(f_i)) = 1$.\smallskip

\noindent\textsc{Claim 2:} By an argument similar to that 
%\footnote{We remark that the proposition itself, as it is stated, does not imply the desired conclusion, as it speaks of strategy profiles allowing for correlation among $i$'s opponents. However, the proof of the proposition could be adopted to show the existence of a strategy profile $\sigma_{-i} \in \Sigma_{-i}$ with the desired property.} 
in Mashiah-Yaakovi \cite[Proposition 11]{Ayala15} or Flesch, Herings, Maes, and Predtetchinski \cite[Theorem 5.7]{JJJJ}, one shows that there is a strategy profile $\sigma_{-i} \in \Sigma_{-i}$ such that 
\[\E_{\sigma_{-i},\sigma_i}(f_{i} \mid h) \leq v_{i}(f_{i,h}) + \ep/2 = v_{i}(f_{i}) + \ep/2,\]
for each history $h \in H$ and each strategy $\sigma_{i} \in \Sigma_{i}$. Fix any $\sigma_{i} \in \Sigma_{i}$. The rest of the proof of the claim is similar to that of Claim 1.
\end{proof}
\end{proposition}

\noindent\textbf{The proof of Theorem \ref{theorem:minmax_indt}:} Fix an $\ep > 0$. By Proposition \ref{prop:v(Q)=1}, $v_{i}(Q_{i,\ep}(f_{i})) = 1$. 

Let $\Gamma^\ep  = (I, A, H, (1_{Q_{i, \ep}(f_i)})_{i \in I})$ be an auxiliary Blackwell game where player $i$'s winning set is $Q_{i, \ep}(f_i)$, the set of player $i$'s $\ep$-individually rational plays in $\Gamma$. Each player's minmax value in the game $\Gamma^\ep$ equals $1$. Therefore, the auxiliary game $\Gamma^\ep$ satisfies the hypothesis of Theorem \ref{theorem:sumofprob}. We conclude that the intersection $\bigcap_{i \in I}Q_{i, \ep}(f_i)$ is not empty, and hence there is a play $p^* \in [H]$ such that $f_i(p^*) \geq v_i(f_i) - \ep$, for every $i \in I$.

The following strategy profile is a $2\ep$-equilibrium of $\Gamma$ (see also Aumann and Shapley \cite{Aumann94}):
\begin{itemize}
\item The players follow the play $p^*$, until the first stage in which one of the players deviates from this play. Denote by $i$ the minimal index of a player who deviates from $p^*$ at that stage.
\item From the next stage and on, the players in $-i$ switch to a strategy profile that reduces player~$i$'s payoff to $v_i(f_i) + \ep$. A strategy profile with this property does exist by the assumption of history-independence of the minmax values.  
\end{itemize}
This completes the proof of the theorem. $\Box$\bigskip

We illustrate the construction of the $\ep$-equilibrium with the following example.

\begin{exl}\label{exl.eq}\rm
We consider a 2-player Blackwell game with history-independent action sets where the same stage game is being played at each stage, and a player's objective is to maximize the long-term frequency of the stages she wins. Specifically, $\Gamma = (\{1,2\}, A_{1}, A_{2}, f_{1}, f_{2})$, where $A_1$ and $A_2$ are finite, and 
\[f_{i}(a_{0},a_{1},\ldots) = \limsup_{t \to \infty}\tfrac{1}{t} \cdot \#\{k < t : a_{k} \in U_{i}\},\]
for each $(a_{0},a_{1},\ldots) \in A^{\N}$. Here $U_i$ is player $i$'s stage winning set. We assume that $U_1$ and $U_2$ are disjoint, and let $d_i$ denote player $i$'s minmax value in the stage game.

Note that $f_i$ is a tail function (see Section  \ref{secn.tail}), and it is a limsup function (in the sense of the definition in Section \ref{secn.approx}). We have $d_{i} = v_{i}(f_{i})$, i.e., player~$i$'s minmax value in the stage game is also player~$i$'s minmax value in the Blackwell game.

Take any Nash equilibrium $x \in \prod_{i \in I}\Delta(A_{i})$ of the stage game. Playing $x$ at each stage is certainly a $0$-equilibrium of the Blackwell game $\Gamma$, but typically it is \textbf{not} of the type that appears in the proof of Theorem \ref{theorem:minmax_indt}. An important feature of the $\ep$-equilibrium constructed in the proof is that the equilibrium play is pure; only off the equilibrium path might a player be requested to play a mixed action. In this particular example we can even choose the equilibrium play to be periodic. This can be done as follows.

First note that $d_{1} + d_{2} \leq 1$. This follows since $\PP_{x}(U_{1}) + \PP_{x}(U_{2}) \leq 1$ (because $U_1$ and $U_2$ are disjoint by supposition) and since $d_i \leq \PP_{x}(U_{i})$ for $i=1,2$ (because the Nash equilibrium payoff is at least the minmax value). Let $\ep > 0$. Choose natural numbers $m$, $m_{1}$, and $m_{2}$ such that $d_{i} - \ep \leq \tfrac{m_{i}}{m} \leq d_{i}$, for $i=1,2$. Note that $m_{1} + m_{2} \leq m$. Pick a point $a_{1} \in U_{1}$ and a point $a_{2} \in U_{2}$, and let $p^*$ be the periodic play with period $m_1+m_2$ 
obtained by repeating $a_{1}$ for the first $m_{1}$ stages, and repeating $a_{2}$ for the next $m_{2}$ stages. 
We have 
\[d_{i} - \ep \leq \tfrac{m_{i}}{m} \leq \tfrac{m_{i}}{m_{1} + m_{2}} = f_{i}(p^*),\]
for $i \in \{1,2\}$. One can support $p^*$ as an $\ep$-equilibrium play by a threat of punishment: in case of a deviation by player $1$, player $2$ will switch to playing the minmax action profile from the stage game for the rest of the game, thus reducing $i$'s payoff to $d_{i}$. A symmetric punishment is imposed on player 2 in case of a deviation.

Under the periodic play $p^*$, the sum of the players' payoffs is $1$. There are alternative plays where the payoff to \textit{both} players is 1, which can support 0-equilibria.
For example, consider the non-periodic play $p$ that is played in blocks of increasing size: for each $k \in \N$, the length of block $k$ is $2^{2^k}$. In even (resp.~odd) blocks the players play the action profile $a_1$ (resp.~$a_2$). The reader can verify that since the ratio between the length of block $k$ and the total length of the first $k$ blocks goes to $\infty$, the payoff to both players at $p$ is 1.
\end{exl}

\section{Regularity and the folk theorem}\label{secn.folk}
A payoff vector $w\in \mathbb{R}^{|I|}$, assigning a payoff to each player, is called \emph{an equilibrium payoff} of the Blackwell game $\Gamma$ if for every $\varepsilon>0$ there exists an $\varepsilon$-equilibrium $\sigma^{\varepsilon}$ of $\Gamma$ such that $\|w-\mathbb{E}_{\sigma^{\varepsilon}}(f)\|_{\infty}\leq \varepsilon$. In other words, an equilibrium payoff is an accumulation point of $\ep$-equilibrium payoff vectors as $\ep$ goes to $0$. We let $\mathcal{E}$ denote the set of equilibrium payoffs. Our goal here is to provide a description of $\mathcal{E}$. 

In repeated games with stage payoffs,
where the total payoff is some average (discounted average with low discounting, liminf of average, limsup of average, etc.) of the stage game payoffs,
the folk theorem states that the set of equilibrium payoffs coincide
with the set of all individually rational vectors
that are in the convex hull of the feasible payoff vectors,
see, e.g., Aumann and Shapley \cite{Aumann94}, Sorin \cite{Sorin92}, and Mailath and Samuelson \cite{{Mailath06}}. As we will see, when the payoff functions are general, the set of equilibrium payoffs is the convex hull of the set of feasible payoff vectors that are individually rational. The reason for the difference is that in repeated games with stage payoffs, getting a low payoff in one stage can be compensated by getting a high payoff in the following stage; when the payoff is obtained only at the end of the game, there is no opportunity to compensate low payoffs.

Define
\begin{align*}
Q^\varepsilon(f)&:=\bigcap_{i\in I} Q_{i, \ep}(f_i),\\
W^\varepsilon(f)&:=\{f(p):p\in Q^\varepsilon(f)\}.
\end{align*}

The set $Q^\varepsilon(f)$ is the set of $\ep$-individually rational plays, and $W^\varepsilon(f)$ is the set of feasible and $\ep$-individually rational payoffs vectors.
Whenever convenient, we write simply $Q^\ep$ and $W^\ep$. For every set $X$ in a Euclidean space we denote its closure by $\textnormal{cl}(X)$ and its convex hull by $\textnormal{conv}(X)$.

\begin{theorem}\label{thrm:folk}
Consider a Blackwell game $\Gamma = (I, A, H, (f_{i})_{i \in I})$. Suppose that for each player $i \in I$, player $i$'s action set $A_i(h)$ at each history $h \in H$ is finite, her payoff function $f_i$ is bounded and upper semi-analytic, and her minmax value is history-independent. Then
\[\mathcal{E}=\bigcap_{\ep>0}\textnormal{conv}(\textnormal{cl}(W^\ep(f))).\]
\end{theorem}

%The proof makes use of the following lemma.

%\begin{lemma}\label{lem.commute}
%It holds that $\textnormal{conv}(\textnormal{cl}(W^\ep)) = \textnormal{cl}(\textnormal{conv}(W^\ep))$. 
%\end{lemma}
%\begin{proof}
%In one direction, since $W^\ep \subseteq \textnormal{conv}(W^\ep)$, we have $\textnormal{cl}(W^\ep)\subseteq \textnormal{cl}(\textnormal{conv}(W^\ep))$, and $\textnormal{cl}(\textnormal{conv}(W^\ep))$ is already convex, establishing the inclusion $\subseteq$. In the other direction, since $\textnormal{cl}(W^\ep) \supseteq W^\ep$, we have $\textnormal{conv}(\textnormal{cl}(W^\ep)) \supseteq \textnormal{conv}(W^\ep)$. Moreover, because $W^\ep$ is bounded, the set $\textnormal{cl}(W^\ep)$ is compact. And the convex hull of a compact set is compact, implying $\supseteq$.
%\end{proof}

To prove Theorem~\ref{thrm:folk} we need the following result, which states that every $\ep$-equilibrium assigns high probability to plays in $Q^{\ep^{1/3}}(f)$.

\begin{lemma}\label{lemma:largeprob}
Consider an $n$-player Blackwell game $\Gamma = (I, A, H, (f_{i})_{i \in I})$. Suppose that for each player $i \in I$, player $i$'s action set $A_i(h)$ at each history $h \in H$ is finite, her payoff function $f_i$ is bounded and upper semi-analytic, and her minmax value is history-independent. Let $\ep > 0$ be sufficiently small, and let $\sigma^\ep$ be an $\ep$-equilibrium. Then 
\[\mathbb{P}_{\sigma^{\ep}}(A^{\N} \setminus Q^{\ep^{1/3}}(f)) < n\ep^{1/3}.\]
\end{lemma}
\begin{proof} 

Set $\eta := \ep^{1/3}$. It suffices to show that for every $i \in I$,
\begin{equation}\label{equ:231}
\mathbb{P}_{\sigma^{\ep}}(A^{\N} \setminus Q_{i, \eta} (f_i)) < \eta.
\end{equation}

Fix a player $i \in I$ and suppose to the contrary that Eq.~\eqref{equ:231} does not hold. 
We derive a contradiction by showing that player~$i$ has a deviation from $\sigma^{\ep}$ that yields her a gain higher than $\ep$.  

For $t \in \N$, denote by $X_t := \mathbb{P}_{\sigma^{\ep}} (Q_{i,\eta}(f_i) |\mathcal{F}_t)$ the conditional probability of the event $Q_{i,\eta}(f_i)$ under the strategy profile $\sigma^\ep$ given the sigma-algebra $\mathcal{F}_t$. By Doob's martingale convergence theorem, $(X_{t})_{t \in \N}$ converges to the indicator function of the event $Q_{i,\eta}(f_i)$, almost surely under $\PP_{\sigma^{\ep}}$. Since by supposition $\mathbb{P}_{\sigma^{\ep}}(A^{\N} \setminus Q_{i, \eta} (f_i)) > \eta$, we know that $\PP_{\sigma^\ep}(X_{t} \to 0) > \eta$. 

Let $K$ be a bound on the game's payoffs, and let $\rho := \ep^2/K$. Let us call a history $h \in H_{t}$ a \textit{deviation history} if under $h$, stage $t$ is the first one such that $X_t < \rho$. On the event $\{X_{t} \to 0\}$, a deviation history arises at some point during play. Consequently, under $\PP_{\sigma^{\ep}}$, a deviation history arises with probability of at least $\eta$. 

Consider the following strategy $\sigma_{i}'$ of player $i$: play according to $\sigma_i^\ep$ until a deviation history, say $h$, occurs (and forever if a deviation history never occurs). At $h$, switch to playing a strategy which guarantees player $i$ a payoff of at least $v_i(f_i) - \ep$ against $\sigma_{-i}^{\ep}$ in $\Gamma_h$. Such a strategy exists by our supposition of history-independence of the minmax values. To conclude the argument, we compute the gain from the deviation to $\sigma_{i}'$. 

For every deviation history $h \in H_{t}$,
\begin{eqnarray}\label{equ:232}
\mathbb{E}_{\sigma_{-i}^{\ep},\sigma_{i}'}(f_{i}\mid h) &\geq& v_{i}(f_i) - \ep,\\
\mathbb{E}_{\sigma_{-i}^{\ep},\sigma_{i}^{\ep}}(f_{i}\mid h) &\leq& \rho K + (1- \rho)(v_i(f_i) - \eta).
\label{equ:233}
\end{eqnarray}
Eq.~\eqref{equ:232} holds by the choice of $\sigma_{i}'$. To derive Eq.~\eqref{equ:233}, suppose that, following the history $h$, player $i$ conforms to $\sigma_{i}^{\ep}$. Then, conditional on $h$, with probability at most $\rho$ the play belongs to $Q_{i,\eta}(f_i)$, and player $i$'s payoff is at most $K$, and with probability at least $1-\rho$ the play does not belong to $Q_{i,\eta}(f_i)$, and player~$i$'s payoff is at most $v_i(f_i)-\eta$. 

We can now compute the gain from the deviation to $\sigma_{i}'$: If a deviation history never arises, $\sigma_{i}'$ recommends the same actions as $\sigma_{i}^{\ep}$, and therefore the gain is 0.
A deviation history occurs with a probability of at least $\eta$, and thus
\begin{align*}
\mathbb{E}_{\sigma_{-i}^{\ep},\sigma_{i}'}(f_{i}) - \mathbb{E}_{\sigma_{-i}^{\ep},\sigma_{i}^{\ep}}(f_{i}) &\geq 
\eta \bigl(v_i(f_i)-\ep - \rho K - (1 - \rho)(v_i(f_i)-\eta)\bigr)\\ 
&= \eta(-\ep - \ep^2 + \rho v_i(f_i) + \eta - \rho \eta)\\
&= \ep^{\frac{2}{3}}(1 - \ep^{\frac{2}{3}} - \ep^{\frac{5}{3}} + \tfrac{v_i(f_i)}{K} \ep^{\frac{5}{3}} - \tfrac{1}{K} \ep^{2}),
\end{align*}
which behaves like $\ep^{\frac{2}{3}}$ when $\ep$ is small, and therefore exceeds $\ep$. \end{proof}

\noindent\textbf{Proof of Theorem~\ref{thrm:folk}}: Let $|I| = n$. Let $w \in \mathbb{R}^{n}$ be an equilibrium payoff. Assume by contradiction that there is an $\alpha > 0$ such that $w \not\in \textnormal{conv}(\textnormal{cl}(W^{\alpha}))$. For a vector $z \in \mathbb{R}^{n}$ write ${\rm dist}(z)$ to denote the distance from $z$ to the set $\textnormal{conv}(\textnormal{cl}(W^{\alpha}))$ under the $\|\cdot\|_{\infty}$ metric on $\mathbb{R}^{n}$. 
By assumption, $\delta := \tfrac{1}{4}{\rm dist}(w) > 0$. Denote $\ep: = \min(\delta, \alpha^3, (\frac{\delta}{Kn})^3) > 0$, where $K$ is a bound on the game payoff.

From $w$ being an equilibrium payoff, there exits an $\ep$-equilibrium, say $\sigma^\ep$, such that $\|w-\mathbb{E}_{\sigma^\ep}(f)\|_{\infty} \leq \ep \leq \delta$. We have the following chain of inequalities:
%\begin{align*}
%{\rm dist}(\mathbb{E}_{\sigma^\ep}(f)) &\leq \mathbb{E}_{\sigma^\ep}({\rm dist}(f))\\ &= \int_{p \in A^{\N} \setminus Q^{\alpha}} {\rm dist}(f(p))d\PP_{\sigma^\ep}(p)\\ &\leq \PP_{\sigma^\ep}(A^{\N} \setminus Q^{\alpha}) \cdot 2K\\ &\leq \ep^{\frac{1}{3}} \cdot 2K \leq 2\delta,
%\end{align*}
\[{\rm dist}(\mathbb{E}_{\sigma^\ep}(f)) \leq \mathbb{E}_{\sigma^\ep}({\rm dist}(f)) \leq 2K \cdot \PP_{\sigma^\ep}(A^{\N} \setminus Q^{\alpha}(f)) \leq 2K \cdot n \cdot  \ep^{\frac{1}{3}} \leq 2\delta,\]
where the first inequality follows from the fact that ${\rm dist}:\mathbb{R}^{n} \to \R$ is a convex function, the second from the fact that $f(p) \in W^{\alpha}$ whenever $p \in Q^{\alpha}(f)$, the third follows since $Q^{\ep^{1/3}}(f) \subseteq Q^{\alpha}(f)$ and   
by Lemma~\ref{lemma:largeprob}, and the last holds by the choice of $\ep$.
But then 
\[{\rm dist}(w) \leq \|w-\mathbb{E}_{\sigma^\ep}(f)\|_{\infty} + {\rm dist}(\mathbb{E}_{\sigma^\ep}(f)) \leq 3\delta,\]
contradicting the choice of $\delta$.

We turn to prove the other direction. Let $w\in \bigcap_{\ep>0}\textnormal{conv}(\textnormal{cl}(W^\ep))$. We need to show that $w$ is an equilibrium payoff. Fix an $\ep>0$. 

Carath\'eodory's Theorem (Carath\'eodory, \cite{Caratheodory07}) implies that $\textnormal{cl} (\textnormal{conv}(W^{\ep})) = \textnormal{conv} (\textnormal{cl}(W^{\ep}))$, hence $w$ is an element of $\textnormal{cl} (\textnormal{conv}(W^{\ep}))$, and thus we can choose a vector $w_{\ep}\in \textnormal{conv}(W^{\ep})$ such that $\|w-w_{\ep}\|_\infty \leq \ep$. We argue that $w_{\ep}$ is a vector of expected payoffs in some $3\ep$-equilibrium.

The payoff $w_{\ep}$ can be presented as a convex combination of $n + 1$ vector payoffs, say \linebreak $f(p^1), \ldots, f(p^{n+1})$, with each $p^{k}$ an element of $Q^{\ep}(f)$. Using jointly controlled lotteries as done, e.g., in Forges \cite{Forges}, Lehrer \cite{Lehrer1996}, or Lehrer and Sorin \cite{LehrerSorin1997}, the players can generate the required randomization over the plays $p^1, \dots, p^{n+1}$ during the first stages of the game. Once a specific play $p^k$ has been chosen, the construction of the $3\ep$-equilibrium is standard: the players play $p^k$, and if player $i$ deviates, her opponents revert to playing a strategy profile that gives player $i$ at most $v_i(f_i)+\ep$. Such a strategy exists by the assumption of history-independence of the minmax values. $\Box$\medskip

\begin{exl}\label{exl.folk}\rm
Consider the 2-player Blackwell game $\Gamma = (\{1,2\}, A_{1}, A_{2}, f_{1}, f_{2})$, where the action sets are $A_1=\left\{{\rm T},{\rm M},{\rm B}\right\}$ and $A_2=\left\{{\rm L},{\rm C},{\rm R}\right\}$, and for a play $p = (a_0,a_1,\ldots)$ the payoffs are
\[(f_1(p),f_2(p)) = \begin{cases}
(1,1)&\text{if }\displaystyle\liminf_{n\to \infty} \tfrac{1}{t} \cdot \#\{k < t: a_k = ({\rm T},{\rm L})\text{ or } a_k = ({\rm M},{\rm C})\} >\frac{1}{2},\\
(4,-1)&\text{if }\displaystyle\liminf_{n\to \infty} \tfrac{1}{t} \cdot \#\{k < t: a_k = ({\rm B},{\rm L})\} = 1,\\
(-1,4)&\text{if }\displaystyle\liminf_{n\to \infty} \tfrac{1}{t} \cdot \#\{k < t: a_k = ({\rm T},{\rm R})\} = 1,\\
(0,0)&\text{otherwise}.
\end{cases}\]
Thus the payoff is $(1,1)$ if the (liminf) frequency of the stages where either (T,L) or (M,C) is played is larger than $\tfrac{1}{2}$. It is $(4,-1)$ if (B,L) is played with frequency of 1, and $(-1,4)$ if (T,R) is played with the frequency of 1. All other cases result in a payoff of $(0,0)$.

%\begin{figure}\centering
%\begin{tabular}{llll}
%                       & L                         & \multicolumn{1}{c}{C}  & R                     \\ \cline{2-4} 
%\multicolumn{1}{l|}{T} & \multicolumn{1}{l|}{*}    & \multicolumn{1}{l|}{}  & \multicolumn{1}{l|}{-1,4} \\ \cline{2-4} 
%\multicolumn{1}{c|}{M} & \multicolumn{1}{l|}{b}    & \multicolumn{1}{l|}{*} & \multicolumn{1}{l|}{}     \\ \cline{2-4} 
%\multicolumn{1}{l|}{B} & \multicolumn{1}{l|}{4,-1} & \multicolumn{1}{l|}{}  & \multicolumn{1}{l|}{}     \\ \cline{2-4} 
%\end{tabular}
%\caption{\label{tab:table}The payoffs in Example~\ref{exl.folk}. If the cells marked with * are played with liminf frequency of more than half, the payoff is $(1,1)$; if (B,L) is played with liminf frequency of 1, the payoff is $(4,-1)$ and it is $(-1,4)$ if the liminf frequency of (T,R) is 1. All other cases result in a payoff of $(0,0)$.}
%\end{figure}

Observe that when player 1 plays B repeatedly, the maximal payoff that player 2 can achieve is 0, and this is player 2's minmax value. Similarly, player 1's minmax value is 0. For each $\ep \in (0,1)$, the set $W^{\ep}(f)$ consists of the two points $(0,0)$ and $(1,1)$. By Theorem~\ref{thrm:folk}, the set of equilibrium payoffs $\mathcal{E}$ is the line segment connecting $(0,0)$ and $(1,1)$, see Figure \ref{figure}. 

Naturally, all equilibrium payoffs $w$ are (a) convex combinations of the feasible payoffs vectors $(0,0)$, $(1,1)$, $(4,-1)$, and $(-1,4)$, and (b) individually rational, i.e., they satisfy $w_{1} \geq 0$ and $w_{2} \geq 0$. The set of all payoff vectors satisfying (a) and (b) is represented in Figure~\ref{figure} by the shaded triangle. The point we wish to make here is that the properties (a) and (b) are not sufficient for a payoff vector to be an equilibrium payoff. 

Take for concreteness the point $(3,0)$. This payoff vector is in the convex hull of the feasible payoff vectors and is individually rational. Yet, for $\ep < \tfrac{2}{3}$, there is no $\ep$-equilibrium with the payoff (close to) the vector $(3,0)$. We give a heuristic argument.

Suppose to the contrary that $\sigma$ is such an $\ep$-equilibrium. The strategy profile $\sigma$ necessarily assigns a probability of at least $\tfrac{2}{3}$ to the set of plays that yield the payoff vector $(4,-1)$. But this implies that Player 2 has a deviation that would improve her payoff over the candidate $\ep$-equilibrium by at least $\tfrac{2}{3}$. Player 2 needs to deviate to playing R forever (for example), at any history of the game where her conditional expected payoff under $\sigma$ is close enough to $-1$. Since playing R would yield at least $0$, by such a deviation, she would improve her conditional expected payoff by at least $1$. Levy's zero-one law guarantees that the histories where player 2 is called to deviate in this way arise with a probability close to $\tfrac{2}{3}$, so that the expected gain from the deviation is also close to $\tfrac{2}{3}$.

%\begin{table}[h]
%\begin{center}
 % \begin{tabular}{| l | c  r  r | }
  %  \hline
   %   &L & C &R \\ \hline
%    T & 3,3 & 0,4 & 0,-2\\
 %   M & -2,-1 & 4,0  & -2,-2\\
  %  B & 5,-1  & -2,0  & -3,-2\\
   % \hline
  %\end{tabular}
%\end{center}
%\caption{The payoff matrix }\label{table 1}
%\end{table}

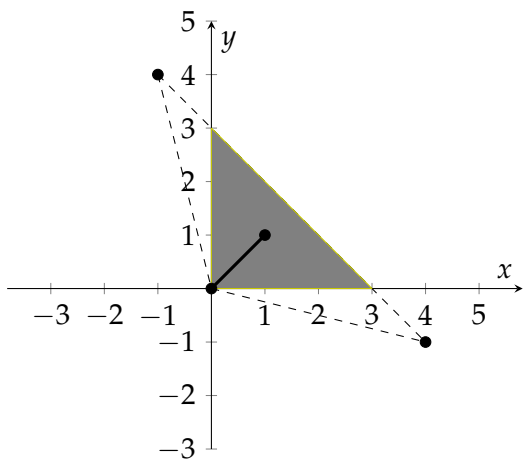
\begin{figure}
\begin{center}
\begin{tikzpicture}
	\begin{axis}[my style, xtick={-3,-2,...,5},ytick={-3,-2,...,5}, xmin=-3, xmax=5, ymin=-3, ymax=5];
	\addplot[ domain=0:1]{x};
	\addplot[dashed,domain=-1:0]{-4*x};
	\addplot[dashed,domain=-0:4]{(-0.25)*x};
	\addplot[dashed,domain=-1:4]{-x+3};
	\addplot[mark=*, only marks] coordinates {(1,1) (0,0) (4,-1)(-1,4)};
\addplot[patch,gray,patch type=triangle] 
	coordinates {
		(0,0) (3,0) (0,3)
	};
	\addplot[very thick, domain=0:1]{x};
	\end{axis}
\end{tikzpicture}
\caption{The set of equilibrium payoffs (the segment connecting $(0,0)$ and $(1,1)$) vs. the set of convex combinations of feasible payoffs that are individually rational (the dark triangle).}\label{figure}
\end{center}
\end{figure}
\end{exl}

The above discussion of Example~\ref{exl.folk} leads to a slightly more general conclusion: if the set of feasible payoffs is finite, then the set of equilibrium payoffs is the convex-hull of the feasible payoffs that are individually rational (equal or larger than the minmax). 
For each player the minmax value is within the finite set of feasible payoffs, and placing any probability on a payoff that is not individually rational enables profitable deviations.

%To facilitate randomizations over the set of payoffs, we use the first stages of the game for a jointly controlled lottery. Solan, Solan and Solan \cite{SolanSolan} show how players can jointly generate any distribution over a finite set of results in a finite number of stages where no player can unilaterally alter the distribution\footnote{To be accurate, any distribution can be jointly generated accurately where the generating process terminates in a finite time with probability 1. In addition, for any $\varepsilon$ there exists a duration of the interaction required to generate a $\varepsilon$ approximation of the distribution.}. Using jointly controlled randomizations any payoff that is in the convex-hull of payoffs that are feasible and individually rational can be an $\varepsilon$-equilibrium payoff (in expectation). 

\section{Blackwell games with tail-measurable payoffs}\label{secn.tail}

An important class of games with history-independent minmax values are those where the payoff functions are tail-measurable. 
In this section we concentrate on games with tail-measurable payoffs.

Consider a Blackwell game with history-independent action sets, $\Gamma = (I,(A_{i},f_{i})_{i \in I})$. A set $Q \subseteq A^{\N}$ is said to be a \textit{tail set} if whenever a play $p = (a_{0},a_{1},\ldots)$ is an element of $Q$ and $q = (b_{0},b_{1},\ldots)$ is such that $a_t = b_t$ for all $t \in \N$ sufficiently large, then $q$ is also an element of $Q$. Let $\mathscr{T}$ denote the sigma-algebra of the tail subsets of $A^{\N}$. We note that the tail sigma-algebra $\mathscr{T}$ and the Borel sigma-algebra $\mathscr{B}$ are not nested. For constructions of tail sets that are not Borel, see Rosenthal \cite{Rosenthal75} and Blackwell and Diaconis \cite{Diaconis96}. 

Examples of tail sets are: (1) the winning sets of Example \ref{exl.io}, (2) the set of plays in which a certain action profile $a\in A$ is played with limsup-frequency at most $\tfrac{1}{2}$, and (3) the set of plays in which a certain action profile $a^*\in A$ is played at most finitely many times at even stages (with no restriction at odd stages).

An important class of tail sets are the shift invariant sets. A set $Q \subseteq A^{\N}$ is a \textit{shift invariant set} if for each play $p = (a_0,a_1,\ldots)$, $p \in Q$ if and only if $(a_1,a_2,\ldots) \in Q$. Equivalently, shift invariant sets are the sets that are invariant under the backward shift operator on $A^{\N}$. Shift invariant sets are tail sets. The converse is not true: while the sets in examples (1) and (2) above are shift invariant, that of example (3) is not. 

A function $f:A^{\N} \to \mathbb{R}$ is called tail-measurable if, for each
$r\in\mathbb{R}$, the set $\{p\in A^{\N} : r \leq f(p)\}$ is an element of $\mathscr{T}$. Intuitively, a payoff function is tail measurable if an action taken in any particular stage of the game has no impact on the payoff. The payoff function in Example \ref{exl.eq} is tail-measurable.

\begin{remark}\rm
The assumption that the set of actions of each player is history-independent is required so that the tail-measurability of the payoff functions has a bite. 
If the sets of actions were history-dependent,
then by having a different set of actions at each history,
any function could be turned into tail-measurable.
\end{remark}

We now state one key implication of tail-measurability, namely the history-independence of minmax values. 

\begin{proposition}\label{prop:minmaxtail}
Let $\Gamma = (I,(A_{i},f_{i})_{i \in I})$ be a Blackwell game with history-independent action sets, and let $i \in I$ be a player. If player $i$'s payoff function is bounded, upper semi-analytic, and tail-measurable, then her minmax value is history-independent.
\end{proposition}
\begin{proof}
It suffices to show that $v_{i}(f_{i,a}) = v_{i}(f_{i})$ for each $a \in A$, where, with a slight abuse of notation, we write $a$ for a history in stage $1$. Since $f_i$ is tail-measurable, all the functions $f_{i,a}$ for $a \in A$ are identical to each other. Hence, fixing any particular action profile $\bar{a} \in A$, letting $X_i := \Delta(A_i)$ and $X_{-i} := \prod_{j \in -i}X_{j}$, we have
\begin{align*}
v_{i}(f_{i}) &= \inf_{x_{-i} \in X_{-i}}\sup_{x_{i} \in X_{i}} \sum_{a \in A} \prod_{j \in I}x_{j}(a_{j}) \cdot \Big(\inf_{\sigma_{-i} \in \Sigma_{-i}}\sup_{\sigma_{i} \in \Sigma_{i}} \E_{\sigma_{-i},\sigma_{i}}(f_{i,a})\Big)\\ 
&= \inf_{x_{-i} \in X_{-i}}\sup_{x_{i} \in X_{i}} \sum_{a \in A} \prod_{j \in I}x_{j}(a_{j}) \cdot v_{i}(f_{i,\bar{a}}) = v_{i}(f_{i,\bar{a}}). 
\end{align*}
\end{proof}

If the payoff functions of all the players in a game $\Gamma$ are tail-measurable, then, for each fixed stage $t \in \N$, all the subgames of $\Gamma$ starting at stage $t$ are identical. On the other hand, the subgames starting, say, at stage $1$, are not identical to the game itself (see example (3) of a tail-measurable payoff function above). Nonetheless, as Proposition \ref{prop:minmaxtail} implies, the players' minmax values \textit{are} the same in every subgame.

The condition of history-independence of the minmax values is more inclusive than that of tail-measurability of the payoffs; the examples that follow illustrate the point.

\begin{exl}\rm 
Consider a one-player Blackwell game where the player's payoff function is $1_{S}$, the indicator of a set $S \subseteq [H]$. 
If $S$ is dense in $[H]$, then the minmax value of the player is $1$ in each subgame. A dense set may or may not be a tail set. 
\end{exl}

\begin{exl}\rm
We consider a Blackwell game similar to that of Example \ref{exl.io}, but where the stage game may depend on the history, as long as each player's stage minmax value is the same.

Specifically, let $\Gamma = (I, A, H, (1_{W_i})_{i \in I})$. Suppose that at each history $h \in H$, each player $i \in I$ has a stage winning set $U_{i}(h) \subseteq A(h)$, and her winning set in the Blackwell game $\Gamma$ is
\[W_i = \{(a_0,a_1,\ldots) \in [H]:a_t \in U_i(a_0,\ldots,a_{t-1})\text{ for infinitely many }t \in \N\}.\]

Assume that the stage minmax value of player $i$ is the same at each history: there is a number $d_{i}$ such that 
\[d_{i} = \inf_{x_{-i} \in \Delta(A_{-i}(h))} \sup_{x_i \in \Delta(A_i(h))} \PP_{x_{-i},x_i}(U_i(h))\]
for every $h \in H$. Then player $i$'s minmax value in each subgame of $\Gamma$ is $0$ if $d_{i} = 0$, and is $1$ if $d_{i} > 0$. Thus player $i$'s minmax value is history-independent. 

Note that the game $\Gamma$ need not have history-independent action sets. Even when the action sets \textit{are} history-independent, the winning sets need not necessarily be tail-measurable. 

To illustrate the last claim, suppose that there are two players playing matching pennies  at each stage. At stage 0, player 1 wants to match the choice of player 2 (and player 2 wants to mismatch the choice of player~1). Subsequently the roles of the two players swap as follows: the player to win stage $t$ wants to match her opponent's action at stage $t+1$, while the loser at stage $t$ wants to mismatch the action of her opponent at stage $t+1$. Formally, we let $\Gamma = (\{1,2\}, A_1, A_2, 1_{W_{1}}, 1_{W_{2}})$ be the 2-player Blackwell game with history-independent action sets, where $A_1 = A_2 = \{{\rm H},{\rm T}\}$, the winning sets $W_1$ and $W_2$ are as above, and the stage winning sets are defined recursively as follows: 
\[U_1(\oslash) = \{({\rm H},{\rm H}), ({\rm T},{\rm T})\}\quad\text{and}\quad
U_2(\oslash) = \{({\rm H},{\rm T}), ({\rm T},{\rm H})\},\]
and
\[U_1(h,a) = \begin{cases}
U_1(\oslash) &\text{if }a \in U_1(h),\\
U_2(\oslash) &\text{if }a \in U_2(h),
\end{cases}\quad\text{and}\quad
U_2(h,a) = \begin{cases}
U_2(\oslash) &\text{if }a \in U_1(h),\\
U_1(\oslash) &\text{if }a \in U_2(h),
\end{cases}\]
for each $h \in H$ and $a \in A$. The sets $W_1$ and $W_2$ are not tail: out of the two plays 
\begin{align*}
&(({\rm H},{\rm H}), ({\rm H},{\rm H}),({\rm H},{\rm H}),\ldots) \hbox{ and}\\
&(({\rm H},{\rm T}), ({\rm H},{\rm H}), ({\rm H},{\rm H}),\ldots),
\end{align*}
the first is an element of $W_1 \setminus W_2$, while the second is an element of $W_2 \setminus W_1$.
\end{exl}

\begin{exl}\rm
Consider a Blackwell game $\Gamma = (I, A, H, (f_{i})_{i \in I})$, where player $i$'s objective is (as in Example \ref{exl.eq}) to maximize the long-term frequency of the stages she wins:  
\[f_{i}(a_{0},a_{1},\ldots) = \limsup_{t \to \infty}\tfrac{1}{t}\cdot\#\{k < t : a_{k} \in U_{i}(a_0,\ldots,a_{k-1})\}.\]
As in the previous example, $U_{i}(h) \subseteq A(h)$ is player $i$'s stage winning set at history $h \in H$. Assume, as above, that player $i$'s minmax value in each stage game is $d_i$. Then also her minmax value in each subgame of $\Gamma$ is $d_{i}$.
\end{exl}

\begin{exl}\rm 
Start with a Blackwell game with tail-measurable payoff functions. Suppose that the minmax values of all the players in the game are $0$. Take any history $h$, and redefine the payoff functions so that any play having $h$ as a prefix has a payoff of $0$. In the resulting game, the minmax value of each player in each subgame remains $0$, but the payoff functions are no longer tail-measurable (unless the original payoff functions are constant). A similar modification can be performed with any subset of histories, not just one.
\end{exl}

From the results above we now deduce a number of implications for Blackwell games with tail-measurable payoffs.

\begin{corollary}\label{cor.0-1law}
Consider a Blackwell game $\Gamma = (I, (A_{i}, 1_{W_{i}})_{i \in I})$ with  history-independent action sets. If player $i$'s winning set $W_i$ is an analytic tail set, then $v_i(W_i)$ is either $0$ or $1$.
\end{corollary}
\begin{proof}
Suppose that $v_i(W_i) > 0$. Let $\ep := v_i(W_i)/2$. In view of Proposition \ref{prop:minmaxtail}, player $i$'s minmax value in $\Gamma$ is history-independent. Applying Proposition \ref{prop:v(Q)=1}, we conclude that $v_i(Q_{i, \ep}(1_{W_{i}})) \linebreak = 1$. But $Q_{i, \ep}(1_{W_{i}}) = W_{i}$ by the choice of $\ep$. 
\end{proof}

The following conclusion follows directly from Proposition \ref{prop:minmaxtail} and Theorem \ref{theorem:minmax_indt}.

\begin{corollary}\label{cor:eq}
Suppose that the game $\Gamma = (I, (A_{i}, 1_{W_{i}})_{i \in I})$ has history-independent action sets. Suppose, furthermore, that for each player $i \in I$, player $i$'s action set $A_i$ is finite and her payoff function $f_i$ is bounded, upper semi-analytic, and tail-measurable. Then for every $\ep>0$ the game admits an $\ep$-equilibrium.
\end{corollary}

\section{Concluding remarks}\label{secn.disc}
\noindent\textbf{Approximations by compact sets.} Any Borel probability measure on 
$[H]$ (recall that $[H]$ is Polish under the maintained assumptions), is not merely regular, but is tight: the probability of a Borel set $B \subseteq [H]$ can be approximated from  below by the probability of a compact subset $K \subseteq B$ (Kechris \cite[Theorem 17.11]{Kechris95}). The minmax value is not tight in this sense. To see this, consider any 2-player Blackwell game where player 1's winning set $W_{1}$ is the entire set of plays $[H]$, so that $v_1(W_1) = 1$,
and where $A_{2}(\o)$, player 2's action set at the beginning of the game, is $\N$. 
We argue that $v_1(K) = 0$ for every compact set $K \subseteq W_1$. Indeed,
the projection of a compact set $K \subseteq W_1$ on $A_{2}(\o)$ is a compact, and hence a finite set. 
Therefore, player 2 can guarantee that the realized play is outside $K$ by choosing a sufficiently large action at stage $0$. Thus $v_1(K) = 0$, as claimed.\medskip

\noindent\textbf{Approximations by semicontinuous functions.} The conclusion of Theorem \ref{thrm:tailapprox} would no longer be true without the assumption of history-indepenence of the minmax values. Here we give an example of a game with a limsup payoff function where the minmax value cannot be approximated from below by an upper semicontinuous function. 

Consider a zero-sum game $\Gamma$ where $A_{1} = A_{2} = \{0,1\}$, and player 1's payoff function is  
\[f(a_0,a_1,\ldots) = \begin{cases}
\displaystyle\limsup_{t \to \infty}\tfrac{1}{t}\#\{k < t: a_{2,k}  = 0\},&\text{if }\tau = \infty,\\
2, &\text{if } \tau < \infty\text{ and }a_{2,\tau} = 1,\\
0, &\text{if } \tau < \infty\text{ and }a_{2,\tau} = 0,   
\end{cases}\]
where $\tau = \tau(a_0,a_1,\ldots) \in \N \cup \{\infty\}$ is the first stage where player 1 chooses action $1$. The game was analyzed in Sorin \cite{Sorin86}, who showed that $v_1(f) = 2/3$.
%Big Match (\cite{Gillette57}), the difference being in the second line of the above equation where the absorbing payoff is $2$ rather than $1$. As in Laraki \cite{Laraki10}, $v_1(f) = 2/3$.

Let $g \leq f$ be a bounded upper semicontinuous function. We argue that $v_1(g) \leq1/2$. 
For $t \in \N$, let $S_t$ denote the set of plays $p$ such that $t \leq \tau(p)$. Note that $S_{t}$ is closed. 
We argue that 
\[\inf_{t \in \N} \sup\{g(p):p \in S_{t}\} \leq 1.\] 
Suppose this is not the case. Take an $\ep > 0$ such that $1 + \ep < \sup\{g(p):p \in S_{t}\}$ for each $t \in \N$. Let $U_{0} := \{1 + \ep \leq g\}$, and for each $t \geq 1$ let $U_{t} := U_{0} \cap S_{t}$. 
The set $U_{t}$ is not empty for each $t \in \N$. Moreover, it is a closed, and hence a compact subset of $A^{\N}$. Thus $U_{0} \supseteq U_{1} \supseteq \cdots$ is a nested sequence of non-empty compact sets. Therefore, there is a play $p \in \bigcap_{t \in \N} U_{t}$. It holds that $\tau(p) = \infty$, and consequently $f(p) \leq 1 < g(p)$, a contradiction.

Take an $\ep > 0$. Find a $t \in \N$ such that $\sup\{g(p):p \in S_{t}\} \leq 1 + \ep$. Suppose that player 2 plays $0$ for the first $t$ stages, and thereafter plays $0$ with probability $1/2$ at each stage. This guarantees that the payoff under the function $g$ is at most $(1 + \ep)/2$.\medskip

\noindent\textbf{On the assumption of finiteness of the action sets.} The hypothesis of Theorem \ref{theorem:sumofprob} requires that the action sets at each history be finite, and its conclusion is not true without this assumption. Indeed, consider the 2-player Blackwell game $\Gamma = (\{1,2\}, A_{1}, A_{2}, W_{1}, W_{2})$ with history-independent action sets $A_1 = A_2 = \N$. Player 1's winning set $W_1$ consists of all plays $(a_{1,t}, a_{2,t})_{t\in\N}$ such that $a_{1,t} > a_{2,t}$ holds for all sufficiently large $t \in \N$, and player 2's winning set $W_2$ consists of all plays $(a_{1,t},a_{2,t})_{t\in\N}$ such that $a_{1,t} < a_{2,t}$ holds for all sufficiently large $t \in \N$. Then $W_1$ and $W_2$ are Borel-measurable and tail-measurable, and $v_1(W_1) = v_2(W_2)=1$, but $W_1\cap W_2=\emptyset$. Hence, the game has no $\ep$-equilibrium for any $\ep < 1/2$. Indeed, an $\ep$-equilibrium $\sigma$ would need to satisfy $\PP_{\sigma}(W_i) \geq v_i(W_i) - \ep > 1/2$ for both $i \in \{1,2\}$.

As discussed above, the assumption that the sets of actions are history-dependent is intertwined with the assumption that the payoffs are tail-measurable.\medskip

\noindent\textbf{Continuity of the minmax.} Unlike Borel probability measures, the minmax value is in general not continuous in the following sense: there is an increasing sequence of Borel sets $C_0 \subseteq C_1\subseteq \ldots$ such that $\lim_{n\to \infty} v_i(C_n) < v_i(\bigcup_{n \in \N} C_n)$. In fact, one can construct an example of this kind where $C_{n}$ is both a $G_{\delta}$ and an $F_{\sigma}$ set, as follows. Consider a 2-player Blackwell game with history-independent action sets where $A_1$ is a singleton (player 1 is a dummy) while $A_2$ contains at least two distinct elements. Let $\{p_{0},p_{1},\ldots\}$ be a converging (with respect to any compatible metric on $A^{\N}$) sequence of plays, no two members of which are the same. Let $C_{n} := A^{\N}\setminus\{p_{n},p_{n+1},\ldots\}$. Then $v_1(C_n) = 0$ for each $n \in \N$ while $v_1(\bigcup_{n \in \N}C_n) = v_1(A^{\N}) = 1$.\medskip

\noindent\textbf{Maxmin value.} Consider a Blackwell game $\Gamma$, and suppose that player $i$'s payoff function $f_i$ is bounded and upper semi-analytic. Player $i$'s \textit{maxmin value} is defined as 
\[z_i(f_i) = \sup_{\sigma_i \in \Sigma_{i}}\inf_{\sigma_{-i} \in \Sigma_{-i}}\mathbb{E}_{\sigma_{-i},\sigma_i}(f_i).\]

The minmax value is not smaller than the maxmin value: $z_i(f_i) \leq v_i(f_i)$. If $I = \{1,2\}$, player 1's payoff function $f_1$ is bounded and Borel-measurable, and for every $h \in H$ either the set $A_1(h)$ of player $1$'s actions or the set $A_2(h)$ of player 2's actions at $h$ is finite, then in fact $z_1(f_1) = v_1(f_1)$, as follows from the determinacy of zero-sum Blackwell games (Martin \cite{Martin98}). Strict inequality might arise for at least two reasons. 

The first is the failure of determinacy. The results of Section \ref{secn.approx} are established under the assumption that the action sets be countable, an assumption that is insufficient to guarantee determinacy of a two-player zero-sum Blackwell game even if  player 1's winning set is clopen. Wald's game provides an illustration. Suppose that each of the two players chooses a natural number; player 1 wins provided that his choice is at least as large as player 2's. Formally, consider a Blackwell game with $I = \{1,2\}$, where the action sets at $\o$ are $A_1(\o) = A_2(\o) = \mathbb{N}$, and player 1's winning set $W_1$ consists of plays such that player 1's stage $0$ action is at least as large as player 2's stage 0 action: $a_{1,0} \geq a_{2,0}$. Then player 1's minmax value is $v_1(W_1) = 1$ while her maxmin value is $z_1(W_1) = 0$.

The second possibility for a maxmin and the minmax values to be different arises in games with three or more players. The reason is that the definitions of both the maxmin and the minmax values impose that the opponents of player $i$ choose their actions independently after each history. The point is illustrated by Maschler, Solan, and Zamir \cite[Example 5.41]{Maschler13}, which can be seen as a 3-player Blackwell game with binary action sets, where the player's payoff function only depends on the stage 0 action profile.

Analogues of Theorems ~\ref{thrm:reg}, \ref{thrm:regfunc}, and \ref{thrm:tailapprox} could be established for the maxmin values using the same approach.\medskip 

\noindent\textbf{Open problems.} Existence of an $\ep$-equilibrium in dynamic games with general \linebreak (Borel-measurable) payoffs has been, and still is, one of the Holy Grails of game theory. A more modest approach, also pursued in this paper, is to establish existence in some special classes of games. Blackwell games, as they are defined here, do not include moves of nature. An interesting avenue for a follow up research is to extend the methods developed in this paper to  the context of stochastic games with general Borel-measurable payoff functions.

Theorems \ref{thrm:reg} and \ref{thrm:regfunc} provide two distinct approximation results, and neither seems to be a consequence of the other. This raises the question of whether there is a natural single generalization that would encompass both these results as two special cases.\medskip  

\section{Appendix: The proof of Theorems~\ref{thrm:reg}, \ref{thrm:regfunc}, and \ref{thrm:tailapprox}}\label{subsecn.proof}
The proofs of Theorems~\ref{thrm:reg} and~\ref{thrm:regfunc} are adaptations of the corresponding arguments in Maitra and Sudderth \cite{Maitra98} and in Martin \cite{Martin98} and are provided here for completeness. Theorem \ref{thrm:tailapprox} follows easily from Theorem~\ref{thrm:reg} and Proposition \ref{prop:v(Q)=1}. 

Consider a Blackwell game $\Gamma = (I, A, H, (f_{i})_{i \in I})$, fix a player $i \in I$, and suppose that player $i$'s payoff function $f_{i}$ is bounded and Borel-measurable. Also assume w.l.o.g. that $0 \leq f_{i} \leq 1$. When we will consider Theorem~\ref{thrm:reg} we will substitute $f_i = 1_{W_i}$.

Given $h \in H$, let $R(h)$ denote the set of one-shot payoff functions $r: A(h) \to [0,1]$. Let $X_i(h) := \Delta(A_i(h))$ denote player $i$'s set of mixed actions at history $h$, and let $X_{-i}(h) := \prod_{j \in -i}X_{j}(h)$. For $x \in \prod_{i \in I}X_{i}(h)$ we write $r(x)$ to denote $\mathbb{E}_{x}(r)$, the expectation of $r$ with respect to $x$. Player $i$'s minmax value of the function $r \in R(h)$ is
\[d_i(r) := \inf_{x_{-i} \in X_{-i}(h)}\sup_{x_i \in X_i(h)} r(x_{-i},x_i).\]

We next introduce the main tool of the proof, an auxiliary two-player game of perfect information denoted by $G_i(f_i,c)$. This is a variation of the games $G_v$ and $G'_v$ in Martin \cite[pp. 1575]{Martin98}.

Given $c \in (0,1]$ and a Borel measurable function $f_i \colon [H] \to [0,1]$, define the game $G_i(f_i,c)$ as follows:
\begin{itemize}
    \item Let $h_0 := \oslash$. Player~I chooses a one-shot payoff function $r_0:A(h_0)\to[0,1]$ such that $d_i(r_0)\geq c$.
    \item Player~II chooses an action profile $a_0 \in A(h_0)$ such that $r_0(a_0) > 0$.
    \item Let $h_1 := (a_0)$. Player~I chooses a one-shot payoff function $r_1:A(h_1)\to[0,1]$ such that $d_i(r_1)\geq r_0(a_0)$.
    \item Player~II chooses an action profile $a_1 \in A(h_1)$ such that $r_1(a_1)>0$.
    \item Let $h_2 := (a_0,a_1)$. Player~I chooses a one-shot payoff function $r_2:A(h_2)\to[0,1]$ such that $d_i(r_2)\geq r_1(a_1)$. And so on.
\end{itemize}
This results in a run\footnote{To distinguish histories and plays of $\Gamma$ from those of $G_i(f_i,c)$, we refer to the latter as \textit{positions} and \textit{runs}. To distinguish the players of $\Gamma$ from those of $G_i(f_i,c)$, we refer to the latter as Player I and Player II, using the initial capital letters.} $(r_0,a_0,r_1,a_1,\ldots)$. Player I wins the run if 
\[\limsup_{t \to \infty}r_{t}(a_{t}) \leq f_{i}(a_0,a_1,\ldots) \quad\text{and}\quad 0 < f_{i}(a_0,a_1,\ldots).\]

Let $T$ be the set of all legal positions in the game $G_i(f_i,c)$. This is a tree on the set $R \cup A$ where $R := \cup_{h \in H}R(h)$. Sequences of even (odd) length in $T$ are Player I's (Player II's) positions. The tree $T$ is pruned: an active player has a legal move at each legal position of the game. Indeed, consider Player I's legal position in the game $G_i(f_i,c)$ and let $h_t$ denote, as above, the sequence of action profiles produced, to date, by Player II. Then the function $r_t$ which is identically equal to $1$ on the set $A(h_t)$ is a legal move for Player I. Consider now Player II's legal position in $G_i(f_i,c)$, let $h_{t}$ denote the sequence of action profiles produced to date by Player II, and let $r_{t}$ be Player I's latest move. Then $d_i(r_t) > 0$. Therefore, there exists an action profile $a_{t} \in A(h_{t})$ such that $r_{t}(a_{t}) > 0$, and thus $a_{t}$ is Player II's legal move at the given position. 

The set $[T]$ is the set of all runs of the game $G_i(f_i,c)$, a subset of $(R \cup A)^{\N}$.

A run is \emph{consistent} with a pure strategy $\sigma_\I$ of Player~I if it is generated by the pair $(\sigma_\I,\sigma_\II)$, for some pure strategy $\sigma_\II$ of Player~II.
Runs that are consistent with pure strategies of Player~II are defined analogously.

Player I's pure strategy $\sigma_{\I}$ in $G_i(f_i,c)$ is said to be \emph{winning} if Player I wins all runs of the game that are consistent with $\sigma_{\I}$. 

\begin{proposition}\label{prop:PlayerI}
Let $c \in (0,1]$ and let $f_i : [H] \to [0,1]$ be a Borel-measurable function. If Player \I~ has a winning strategy in the game $G_i(f_i,c)$, then there exists a closed set $C \subseteq [H]$ and a limsup function $g : [H] \to [0,1]$ such that $g \leq f_i$, $\{g > 0\} \subseteq C \subseteq \{f_i > 0\}$, and $c \leq v_{i}(g)$. In particular, $c \leq v_i(C)$; and if $f_i = 1_{W_i}$, then $C \subseteq W_i$.
\end{proposition}

\begin{proof}
Fix Player I's winning strategy $\sigma_{\I}$ in $G_i(f_i,c)$.\smallskip

\noindent\textsc{Step 1:} Defining $C \subseteq [H]$ and $g : [H] \to [0,1]$.

Let $T_\I \subseteq T$ denote the set of positions in the game $G_i(f_i,c)$ of even length (i.e., Player I's positions) that are consistent with $\sigma_\I$, i.e., those positions that can be reached under a strategy profile $(\sigma_{\I},\sigma_{\II})$ for some  pure strategy of $\sigma_{\II}$ of Player II. Let $\pi_\I : T_\I \to H$ be the projection that maps a position of length $2t$ in $G_i(f_i,c)$ to a history of length $t$ in $\Gamma$: Formally, $\pi_\I(\oslash) := \oslash$, $\pi_\I(r_0,a_0) := (a_0)$, etc. Let $H_\I \subseteq H$ be the image of $T_\I$ under $\pi_\I$. Since in the tree $T_\I$ Player I's moves are uniquely determined by $\sigma_\I$, the map $\pi_\I$ is in fact a bijection between $T_\I$ and $H_\I$. We write $\phi : H_\I \to T_\I$ for the inverse of $\pi_\I$. The map $\phi$ induces a continuous bijection $[H_\I] \to [T_\I]$, which we also denote by $\phi$. 
We say that positions in $H_\I$ are \emph{$\sigma_\I$-acceptable}, and define $C$ to be the set $[H_\I]$.  

For each $t \in \N$, define the function $\rho_t : H_{t} \to \mathbb{R}$ as follows: $\rho_0(\oslash) := c$. Let $t \in \N$ and consider a history $h_{t} \in H_{t}$. If $h_{t}$ is not $\sigma_\I$-acceptable, we define $\rho_{t+1}(h_{t},a_{t}) := 0$ for each $a_{t} \in A(h_{t})$. Suppose that $h_{t}$ is $\sigma_\I$-acceptable, and let $r_{t} := \sigma_\I(\phi(h_{t}))$. For each $a_{t} \in A(h_{t})$ define $\rho_{t+1}(h_{t},a_{t}) := r_{t}(a_{t})$. Note that if $h_{t}$ is $\sigma_\I$-acceptable while $(h_{t},a_{t})$ is not, we have $\rho_{t+1}(h_{t},a_{t}) = r_{t}(a_{t}) = 0$. 

Also define $g : [H] \to [0,1]$ by letting
\[g(a_{0},a_{1},\ldots) := \limsup_{t \to \infty}\rho_t(a_{0},\dots,a_{t-1}).\]
\smallskip

\noindent\textsc{Step 2:} Verifying that $g \leq f_i$ and $\{g > 0\} \subseteq C \subseteq \{f_i > 0\}$.

Since $\sigma_\I$ is Player I's winning strategy in $G_i(f_i,c)$, all runs in $[T_\I]$ are won by Player I, and hence $[H_\I] \subseteq \{f_i > 0\}$. For a play $p = (a_{0},a_{1},\ldots)$ in $[H_\I]$, if $\phi(p) = (r_{0},a_{0},r_{1},a_{1},\ldots)$, then $g(p)$ equals $\limsup_{t \to \infty}r_{t}(a_{t})$. Since the run $\phi(p)$ is won by Player I, we conclude that $g(p) \leq f_i(p)$. Thus $g \leq f_i$ on $[H]$.\smallskip 

\noindent\textsc{Step 3:} Verifying that $c \leq v_{i}(g)$. Since $g \leq 1_{C}$ it will then follow that $c \leq v_i(C)$.

Fix a strategy profile $\sigma_{-i} \in \Sigma_{-i}$ for the players in $-i$ in the game $\Gamma$. Take any $\epsilon > 0$. We define a strategy $\sigma_i$ for player $i$ in the game $\Gamma$ with the property that $\E_{\sigma_{-i},\sigma_{i}}(g) \geq c - 2\epsilon$.\smallskip 

\noindent\textsc{Step 3.1:} Defining player $i$'s strategy $\sigma_{i}$.

Let $r_0 := \sigma_\I(\oslash)$, Player I's first move in $G_i(f_i,c)$ according to her strategy $\sigma_\I$. Define $\sigma_{i}(\oslash)$ to be a mixed action on $A_i(\oslash)$ such that
\[r_0(\sigma_{-i}(\oslash),\sigma_{i}(\oslash)) \geq c - \epsilon.\]
Let $t \geq 1$ and consider a history $h_{t} = (a_{0},\ldots,a_{t-1}) \in H_{t}$ of $\Gamma$. If $h_{t}$ is not $\sigma_\I$-acceptable, then $\sigma_{i}(h_{t})$ is arbitrary. If $h_{t}$ is $\sigma_\I$-acceptable, let $\phi(h_{t}) := (r_{0}, a_{0}, \ldots, r_{t-1}, a_{t-1})$ and $r_{t} := \sigma_\I(\phi(h_{t}))$. Define $\sigma_{i}(h_{t})$ to be a mixed action on $A_i(h_{t})$ such that
\[r_t(\sigma_{-i}(h_{t}),\sigma_{i}(h_{t})) \geq r_{t-1}(a_{t-1})-\epsilon \cdot 2^{-t}.\]

\noindent\textsc{Step 3.2:} Verifying that $\E_{\sigma_{-i},\sigma_{i}}(g) \geq c - 2\epsilon$.

For each $t \in \N$ let us define $\rho_t^{\epsilon} := \rho_t - \epsilon \cdot 2^{-t+1}$. One can think of the functions $\rho_0^{\epsilon}, \rho_1^{\epsilon}, \dots$ as a stochastic process on $[H]$ that is measurable with respect to the filtration $\{\mathcal{F}_{t}\}_{t \in \N}$. We now argue that this process is a submartingale with respect to the measure $\PP_{\sigma_{-i},\sigma_{i}}$. 

Letting $r_0 := \sigma_\I(\oslash)$ we have
\[\mathbb{E}_{\sigma_{-i},\sigma_{i}}(\rho_{1}^{\epsilon}) = \mathbb{E}_{\sigma_{-i},\sigma_{i}}(r_{0}(a_{0})) - \epsilon=
r_{0}(\sigma_{-i}(\oslash),\sigma_{i}(\oslash)) - \epsilon \geq c-2\epsilon = \rho_{0}^{\epsilon}(\oslash).\]
Consider a $\sigma_\I$-acceptable history $h_{t} = (a_{0},\dots,a_{t-1}) \in H_t$ of length $t \geq 1$. Let $(r_{0},a_{0},\ldots,r_{t-1},a_{t-1}) := \phi(h_{t})$ and $r_{t} := \sigma_\I(\phi(h_{t}))$. We have
\begin{align*}
\mathbb{E}_{\sigma_{-i},\sigma_{i}}(\rho_{t+1}^{\epsilon} | h_{t}) &= \mathbb{E}_{\sigma_{-i},\sigma_{i}}(r_{t}(a_{t}) | h_{t}) - \epsilon \cdot 2^{-t}\\ &= r_{t}(\sigma_{-i}(h_{t}),\sigma_{i}(h_{t})) - \epsilon \cdot 2^{-t}\\ &\geq r_{t-1}(a_{t-1}) - \epsilon \cdot 2^{-t} - \epsilon \cdot 2^{-t}\\ &= \rho_t^{\epsilon}(h_{t}).
\end{align*}
On the other hand, if $h_{t}$ is not $\sigma_\I$-acceptable, then
\[\mathbb{E}_{\sigma_{-i},\sigma_{i}}(\rho_{t+1}^{\epsilon} | h_{t}) = -\epsilon \cdot 2^{-t} > -\epsilon \cdot 2^{-t+1} = \rho_t^{\epsilon}(h_{t}).\]
This establishes the submartingale property for $\rho_0^{\epsilon}, \rho_1^{\epsilon}, \dots$.

The submartingale property implies that $\E_{\sigma_{-i},\sigma_{i}}(\rho_t^{\epsilon}) \geq \rho_0^{\epsilon}(\oslash) = c - 2\epsilon$ for each $t \in \N$. Using Fatou lemma we thus obtain 
\[\E_{\sigma_{-i},\sigma_{i}}(g) =\E_{\sigma_{-i},\sigma_{i}}(\limsup_{t \to \infty}\rho_t) \geq \E_{\sigma_{-i},\sigma_{i}}(\limsup_{t \to \infty}\rho_t^{\epsilon}) \geq \limsup_{t \to \infty}\E_{\sigma_{-i},\sigma_{i}}(\rho_t^{\epsilon}) \geq c - 2\epsilon,\]
as desired.
\end{proof}

\begin{proposition}\label{prop:PlayerII}
Let $c \in (0,1]$ and let $f_i : [H] \to [0,1]$ be a Borel-measurable function. If Player \II~has a winning strategy in the game $G_i(f_i,c)$, then for every $\epsilon > 0$ there exists an open set $O \subseteq [H]$ and a limsup function $g : [H] \to [0,1]$ such that $f_i \leq g$, $\{f_i = 1\} \subseteq O \subseteq \{g = 1\}$, and $v_{i}(g) \leq c + \epsilon$. In particular, $v_i(O) \leq c + \epsilon$; and if $f_i = 1_{W_i}$, then $W_i \subseteq O$.
\end{proposition}

\begin{proof}
Fix Player II's winning strategy $\sigma_\II$ in $G_i(f_i,c)$.\smallskip

\noindent\textsc{Step 1:} Defining $O \subseteq [H]$ and $g : [H] \to [0,1]$. 

We recursively define (a) the notion of a \textit{$\sigma_\II$-acceptable} history in the game $\Gamma$, (b) for each $\sigma_\II$-acceptable history $h$ in $\Gamma$, Player~I's position $\psi(h)$ in the game $G_i(f_i,c)$, and (c) for each $\sigma_\II$-acceptable history $h$ of $\Gamma$, a function $u_{h} : A(h) \to [0,1]$.

The empty history $\oslash$ of $\Gamma$ is $\sigma_\II$-acceptable. We define $\psi(\oslash) := \oslash$, the empty history in $G_i(f_i,c)$. Let $t \in \N$ and consider a history $h_{t} \in H_{t}$ of the game $\Gamma$. If $h_{t}$ is not $\sigma_\II$-acceptable, so is the history $(h_{t},a_{t})$ for each $a_{t} \in A(h_{t})$. Suppose that $h_{t}$ is $\sigma_\II$-acceptable and that Player I's position $\psi(h_{t})$ in $G_i(f_i,c)$ has been defined. Take $a_{t} \in A(h_{t})$. Let $R^{*}(h_{t},a_{t})$ denote the set of Player~I's legal moves at position $\psi(h_{t})$ to which $\sigma_\II$ responds with $a_{t}$:
\[R^{*}(h_{t},a_{t}) := \{r_{t} \in R(h_t): (\psi(h_{t}),r_{t}) \in T\text{ and }\sigma_\II(\psi(h_{t}),r_{t}) = a_{t}\}.\] 
The history $(h_{t},a_{t})$ is defined to be $\sigma_\II$\emph{-acceptable} if $R(h_{t},a_{t})$ is not empty. In this case we define
\[u_{h_{t}}(a_{t}) := \inf\{r_{t}(a_{t}): r_{t} \in R^{*}(h_{t},a_{t})\}.\]
Choose $r_{t} \in R^{*}(h_{t},a_{t})$ with the property that
\begin{equation}\label{eqn:proprt}
u_{h_{t}}(a_{t}) \leq r_{t}(a_{t}) \leq u_{h_{t}}(a_{t}) + \epsilon \cdot 3^{-t-2},
\end{equation}
and define $\psi(h_{t},a_{t}) := (\psi(h_{t}),r_{t},a_{t})$.

Finally, extend the definition of $u_{h}$ to all histories $h$ of $\Gamma$ by setting $u_{h}(a) := 1$ whenever $(h,a)$ is not $\sigma_\II$-acceptable. 

Let $H_{\II}$ be the set of $\sigma_\II$-acceptable histories of $\Gamma$. We define the set $O$ to be the complement of $[H_\II]$, that is $O := [H] \setminus [H_\II]$. Since $[H_\II]$ is a closed subset of $[H]$ (e.g. Kechris \cite[Proposition 2.4]{Kechris95}), $O$ is an open subset of $[H]$. Let $T_{\II} \subseteq T$ be the image of $H_\II$ under $\psi$. The function $\psi_{\II} : H_{\II} \to T_{\II}$ induces a continuous function $\psi_{\II} : [H_{\II}] \to [T_{\II}]$. Note that all runs in $[T_{\II}]$ are consistent with Player II's winning strategy $\sigma_\II$.  

For $t \in \N$ define a function $\upsilon_{t} : H_{t} \to \mathbb{R}$ by letting $\upsilon_0(\oslash) := c$; and for each $t \in \N$ and each history $(h_{t},a_{t}) \in H_{t+1}$, by letting $\upsilon_{t+1}(h_{t},a_{t}) := u_{h_{t}}(a_{t})$. Note that, for $t \in \N$ and $h_{t} \in H_{t}$, we have $\upsilon_{t}(h_{t}) = 1$ whenever $h_{t}$ is not $\sigma_\II$-acceptable.  

Also define $g : [H] \to [0,1]$ by letting 
\[g(a_0,a_1,\ldots) := \limsup_{t \to \infty}\upsilon_{t}(a_0,\ldots,a_{t-1}).\]

\noindent\textsc{Step 2:} Verifying that $f_i \leq g \leq 1$ and $\{f_i = 1\} \subseteq O \subseteq \{g = 1\}$.

The function $g$ is equal to $1$ on the set $O$; thus $O \subseteq \{g = 1\}$. Consider a play $p = (a_0,a_1,\ldots) \in [H_\II]$, and let $\psi(p) := (r_{0}, a_{0}, r_{1}, a_{1}, \ldots)$. It follows by \eqref{eqn:proprt} that 
\[g(p) := \limsup_{t \to \infty}r_{t}(a_{t}).\] 
Since the run $\psi(p)$ is won by Player II, it must hold that either $f_i(p) < g(p)$ or $0 = f_i(p)$; in either case $f_i(p) < 1$ and $f_i(p) \leq g(p)$. We conclude that $[H_{\II}] \subseteq \{f_i < 1\}$, or equivalently that $\{f_i = 1\} \subseteq O$, and that $f_i \leq g$ on $[H]$.\smallskip

\noindent\textsc{Step 3:} Verifying that $v_i(g) \leq c + \epsilon$. Since $1_{O} \leq g$, it then follows that $v_i(O) \leq c + \epsilon$.\smallskip

\noindent\textsc{Step 3.1:} Defining a strategy profile for player~$i$'s opponents.

First we argue that
\begin{equation}\label{eqn argue}
d_i(u_{\oslash}) \leq c.
\end{equation}
Suppose to the contrary that $c \leq d_i(u_{\oslash}) - \lambda$ for some $\lambda > 0$. Define $r_0 \in R(\oslash)$ by letting $r_0(a) := \max\{u_{\oslash}(a) - \lambda,0\}$.
Since $u_{\oslash} - \lambda \leq r_0$, it holds that $c \leq d_{i}(u_{\oslash}) - \lambda \leq d_{i}(r_0)$. Consequently, $r_0$ is a legal move of Player~I in the game $G_i(f_i,c)$ at position $\oslash$. Denote $a_0 := \sigma_\II(r_{0})$. As $a_0$ is Player~II's legal move in $G_i(f_i,c)$ at position $(r_0)$, it must be the case that $r_0(a_0) > 0$, and hence $r_0(a_0) = u_{\oslash}(a_{0}) - \lambda$. On the other hand, $r_{0} \in R^{*}(\oslash,a_{0})$, so the definition of $u_{\oslash}$ implies that  $u_{\oslash}(a_{0}) \leq r_{0}(a_{0})$, a contradiction.

Take $t \geq 1$, let $h_{t} := (h_{t-1},a_{t-1}) \in H_{t}$ be a $\sigma_\II$-acceptable history, and let $r_{t-1}$ be such that $\psi(h_{t}) = (\psi(h_{t-1}),r_{t-1},a_{t-1})$. Then
\begin{equation}\label{eqn:argue1}
d_i(u_{h_{t}}) \leq r_{t-1}(a_{t-1}).
\end{equation}

Indeed, suppose to the contrary that $r_{t-1}(a_{t-1}) \leq d_i(u_{h_{t}}) -  \lambda$ for some $\lambda > 0$. Define $r_{t} \in R(h_{t})$ by letting $r_{t}(a) := \max\{u_{h_{t}}(a) - \lambda,0\}$. Since $u_{h_{t}} - \lambda \leq r_{t}$, it holds that $r_{t-1}(a_{t-1}) \leq d_{i}(u_{h_{t}}) - \lambda \leq d_{i}(r_{t})$. Consequently, $r_{t}$ is a legal move of Player~I at position $\psi(h_{t})$. Let $a_{t} := \sigma_\II(\psi(h_{t}), r_{t})$. As $a_{t}$ is Player~II's legal move at position $(\psi(h_{t}), r_{t})$, it must be the case that $r_{t}(a_{t}) > 0$, and hence $r_{t}(a_{t}) = u_{h_{t}}(a_{t}) - \lambda$. On the other hand, $r_{t} \in R^{*}(h_{t}, a_{t})$, so the definition of $u_{h_{t}}$ implies that $u_{h_{t}}(a_{t}) \leq r_{t}(a_{t})$, a contradiction.

We now define a strategy profile $\sigma_{-i}$ of $i$'s opponents in $\Gamma$ as follows: For a history $h_{t} \in H_{t}$ of $\Gamma$ let $\sigma_{-i}(h_{t}) \in X_{-i}(h)$ be such that 
\begin{equation}\label{eqn str opponent}
u_{h_{t}}(\sigma_{-i}(h_{t}),x_{i}) \leq d_{i}(u_{h_{t}}) + \epsilon \cdot 3^{-t-1}\text{ for each }x_{i} \in \Delta(A_i(h_{t})).
\end{equation}
\smallskip

\noindent\textsc{Step 3.2:} Verifying that $\E_{\sigma_{-i},\sigma_{i}}(g) \leq c + \epsilon$ for each strategy $\sigma_{i} \in \Sigma_{i}$ of player $i$ in $\Gamma$.

Fix a strategy $\sigma_{i} \in \Sigma_{i}$. For $t \in \N$ define a function $\upsilon_{t}^\epsilon := \upsilon_{t} + \epsilon \cdot 3^{-t}$. The sequence $\upsilon_0^\epsilon, \upsilon_1^\epsilon, \dots$ could be thought of as a process on $[H]$, measurable with respect to the filtration $\{\mathcal{F}_{t}\}_{t \in \N}$. We next show that the process is a supermartingale w.r.t $\PP_{\sigma_{-i},\sigma_{i}}$.

By Eqs. \eqref{eqn str opponent} and \eqref{eqn argue},
\begin{align*}
\mathbb{E}_{\sigma_{-i},\sigma_{i}}(\upsilon_1^\epsilon) 
&= \mathbb{E}_{\sigma_{-i},\sigma_{i}}(u_{\oslash}(a_0)) + \epsilon\cdot 3^{-1}\\
&=  u_{\oslash}(\sigma_{-i}(\oslash),\sigma_{i}(\oslash)) + \epsilon\cdot 3^{-1}\\
&\leq d_{i}(u_{\oslash}) + \epsilon\cdot 2 \cdot 3^{-1}\\ 
&\leq c + \epsilon = \upsilon_{0}^\epsilon(\oslash).
\end{align*}
Take $t \geq 1$, let $h_{t} = (h_{t-1},a_{t-1}) \in H_{t}$ be a $\sigma_\II$-acceptable history, and let $r_{t-1}$ be such that $\psi(h_{t}) = (\psi(h_{t-1}),r_{t-1},a_{t-1})$. We have by Eqs. \eqref{eqn str opponent}, \eqref{eqn:argue1}, and \eqref{eqn:proprt}: 
\begin{align*}
\mathbb{E}_{\sigma_{-i},\sigma_{i}}(\upsilon_{t+1}^\epsilon \mid h_{t})
&= \mathbb{E}_{\sigma_{-i},\sigma_{i}}(u_{h_{t}}(a_t) \mid h_{t}) + \epsilon\cdot3^{-t-1}\\
&= u_{h_{t}}(\sigma_{-i}(h_{t}),\sigma_{i}(h_{t})) + \epsilon\cdot3^{-t-1}\\
&\leq d_{i}(u_{h_{t}}) + \epsilon \cdot 2 \cdot 3^{-t-1}\\
&\leq r_{t-1}(a_{t-1}) + \epsilon \cdot 2 \cdot 3^{-t-1}\\
&\leq u_{h_{t-1}}(a_{t-1}) + \epsilon \cdot 3 \cdot 3^{-t-1}\\
&= \upsilon_{t}^\epsilon(h_{t-1},a_{t-1}) = \upsilon_{t}^\epsilon(h_{t}).
\end{align*}
If, on the other hand, the history $h_{t}$ is not $\sigma_\II$-acceptable, then 
\[\mathbb{E}_{\sigma_{-i},\sigma_{i}}(\upsilon_{t+1}^\epsilon \mid h_{t}) = 1+\epsilon\cdot 3^{-t-1} \leq 1+\epsilon\cdot 3^{-t} = \upsilon_{t}^\epsilon(h_{t}).\]

Since the process $\upsilon_0^\epsilon,\upsilon_1^\epsilon,\dots$ is bounded below (by 0), by the Martingale Convergence Theorem, it converges pointwise $\PP_{\sigma_{-i},\sigma_{i}}$-almost surely; whenever the process converges, its limit is $g$. Hence $\mathbb{E}_{\sigma_{-i},\sigma_{i}}(g) = \mathbb{E}_{\sigma_{-i},\sigma_{i}}(\lim_{t \to \infty}\upsilon_t^\epsilon) \leq \upsilon_0^\epsilon(\oslash) = c + \epsilon$, as desired. 
\end{proof}

We now invoke the result of Martin \cite{Martin75} on Borel determinacy of perfect information games. To do so, we endow $[T]$ with its relative topology as a subspace of the product space $(R \cup A)^{\N}$, where $R \cup A$ is given its discrete topology. One can then check that Player I's winning set in $G_i(f_i,c)$ is a Borel subset of $[T]$. It follows that  for each $c \in (0,1]$ the game $G_i(f_i,c)$ is determined: either Player~I has a winning strategy in the game or Player~II does. We arrive at the following conclusion.
   
\begin{proposition}\label{prop:det}
If $v_i(f_i) < c$, then Player II has a winning strategy in $G_i(f_i,c)$. If $c< v_i(f_i)$, then Player I has a winning strategy in $G_i(f_i,c)$.
\end{proposition}

Theorems \ref{thrm:reg} and \ref{thrm:regfunc} follow from Propositions \ref{prop:PlayerI}, \ref{prop:PlayerII}, and \ref{prop:det}.\bigskip

\noindent\textbf{Proof of Theorem \ref{thrm:tailapprox}:} Take an $\ep > 0$. Without loss of generality, suppose that $f_i$ takes values in $[0,1]$.

By Proposition \ref{prop:v(Q)=1} we know that $v_{i}(Q_{i,\ep}(f_i)) = 1$. To obtain an approximation from below, use Theorem \ref{thrm:reg} to choose a closed set $C \subseteq Q_{i, \ep}(f_i)$ such that $1 - \ep \leq v_{i}(C)$, and define the function $g := (v_i(f_i) - \ep) \cdot 1_{C}$. Then $g \leq f_i$ and $v_i(f_i) - 2\ep \leq (v_i(f_i) - \ep)\cdot(1-\ep) \leq v_{i}(g)$. Since $C$ is closed, $g$ is upper semicontinuous.

By Proposition \ref{prop:v(Q)=1} we know that $v_{i}(U_{i,\ep}(f_i)) = 0$. To obtain an approximation from above, use Theorem \ref{thrm:reg} to choose an open set $O \supseteq U_i^{\ep}(f_i)$ such that $v_{i}(O) \leq \ep$, and define the function $g := v_i(f_i) + \ep + (1 - v_i(f_i) - \ep) \cdot 1_{O}$. Then $f_i \leq g \leq 1$ and $v_{i}(g) \leq v_i(f_i) + 2\ep$. Since $O$ is open, $g$ is lower semicontinuous. $\Box$

\end{document}